%% file: Oorts_conjecture.tex
\documentclass{amsart} 

\input{preamble}

\begin{document}
\title{Oort's conjecture for split unitary Shimura varieties}

\author{Paul Philippe, Fabian Schnelle, Eva Viehmann}
\address{Universit\"at M\"unster, Einsteinstr. 62, 48149 M\"unster, Germany}

\begin{abstract}
We prove that generically on the basic stratum of split unitary Shimura varieties, the universal abelian variety with endomorphism structure and polarization has minimal automorphism group in a precise sense, except for a few degenerate cases. This is a direct analogue of Oort's conjecture on automorphisms of supersingular abelian varieties. On the way, we explicitly compute the generic automorphism group of the universal $p$-divisible group without extra structure in a given basic isogeny class.
\end{abstract}

\maketitle
\tableofcontents
\section{Introduction}
\subsection{Oort's conjecture}
Fix a prime $p$, an integer $g \geq 1$ and an integer $N \geq 3$ which is coprime to $p$. Let $\ccA_g$ denote the Siegel modular variety at level $N$, which is the Shimura variety associated to the reductive group $\GSp_{2g}$ and which parametrizes principally polarized abelian varieties of dimension $g$ with level $N$-structure. Decomposing it according to the isogeny class of the polarized $p$-divisible group of the universal polarized abelian variety induces the so-called Newton stratification into finitely many locally closed subschemes. 

It is a well-known fact that the automorphism group of a principally polarized abelian variety is finite and contains $\{\pm 1\}$. Thus it is a natural question to ask for which points of $\ccA_g$ the automorphism group of the universal principally polarized abelian variety consists precisely of $\{\pm 1\}$. In \cite{chaioort2011monodromy}, Chai and Oort prove that this is generically true on each individual Newton stratum of $\cA_g$, except for the unique closed Newton stratum $\cA_{g,b_0}$, which is also called the supersingular stratum. Already in 2001, Oort conjectured in \cite[Problem 4]{edixhovenmoonenoort2001problems} that for all $g\geq 2$, the analogous assertion also holds for the supersingular stratum. This has become known as Oort's conjecture on automorphisms of generic supersingular abelian varieties.

Many people have worked on this conjecture, starting with Ibukiyama \cite{ibukiyama2020polarizations} and Karemaker and Pries \cite{karemakerpries2019fullymaximal} for $g=2$, where the conjecture is false for $p=2$ but true otherwise. Various particular cases, for small $g$ or large $p$, have then been proved in articles by Karemaker, Yobuko and Yu \cite{karemakeryobukoyu2021massformula}, Karemaker and Yu \cite{karemakeryu2025supersingular}, and Dragutinovi\'c \cite{dragutinovic2024oort}. The general case has been settled by the third author.

\begin{theorem*}[\cite{viehmann2026oort}, Thm.~1.1]
    Let $(g,p)\neq (2,2),(3,2)$. Then there is a dense open subscheme $Y_{b_0}$ of the supersingular stratum $\cA_{g,b_0}$ such that for every $x\in Y_{b_0}(\overline{\mathbb{F}}_p)$, the automorphism group of the universal principally polarized abelian variety at $x$ equals $\{\pm 1\}$.
\end{theorem*}

Karemaker and Yu \cite{karemakeryu2026} have recently obtained another proof of the conjecture for odd primes. 

The main goal of this article is to study the analogous questions for split unitary Shimura varieties of signature $(m_0,n_0-m_0)$ for integers $n_0 \geq 2$ and $1 \leq m_0 \leq n_0-1$, which we are going to introduce in Section~\ref{Section: Shimura varieties}. In short, they parametrize principally polarized abelian varieties with an additional endomorphism structure, given by a central division algebra over an imaginary quadratic field $L$ which splits at the prime $p$. We denote the special fiber of the split unitary Shimura variety by $\cS$. As before, this admits a Newton stratification into finitely many locally closed strata, which we denote by $\cS_{b}$. Again, there is a unique closed stratum, called the basic stratum and denoted by $\cS_{b_0}$. However, the universal abelian variety over $\cS_{b_0}$ is no longer supersingular unless $n_0=2m_0$.

In the supersingular case, i.e.~for signature $(g,g)$ for $g \geq 2$, some of the results of this paper have already been obtained by the third author in \cite[Section 5]{viehmann2026oort}.

\subsection{Main results}
Before stating our main results, we specify what we mean by automorphism groups in our context.
\begin{definition}
Let $k$ be a field of characteristic $p$, and let $(A,i,\lambda)$ correspond to a $k$-valued point of the split unitary Shimura variety $\cS$. Then we set
\[\Aut(A,i,\lambda):=\{\varphi:A\overset{\sim}{\longrightarrow} A\mid i(b) \circ\varphi=\varphi \circ i(b) \text{ for all } b \in \O_B \text{ and } \varphi^\vee \circ \lambda\circ \varphi=\lambda\}.\]
\end{definition}

Let $(-)^*$ denote the non-trivial Galois action on $L$, and define the norm of an element $\alpha \in L$ as $\Norm(\alpha) := \alpha^*\alpha$. Note that every $\alpha \in \O_L$ that satisfies $\Norm(\alpha) =1$ yields an automorphism $i(\alpha) \in \Aut(A,i,\lambda)$ for any $k$-valued point $(A,i,\lambda)$ of $\cS$ as above. Hence the minimal possible automorphism group for any such $(A,i,\lambda)$ is $\O_L^1$, the (multiplicative) subgroup of all norm $1$ elements of $\O_L$. Our main result is now as follows.
\begin{theorem}\label{thmmain1}
    Assume that the split unitary Shimura datum is of signature $(m_0,n_0-m_0)$ such that either
    \begin{enumerate}
        \item $p \geq 3$ and $2 \leq m_0 \leq n_0-2$ or
        \item $p=2$ and either $3 \leq m_0 \leq n_0-3$ or $n_0$ is odd and $m_0\in\{2,n_0-2\}$.
    \end{enumerate}
    Then there is an open dense subscheme $Y_{b_0}$ of $\mathcal S_{b_0}$ such that 
        \[\Aut(A_y,i_y,\lambda_y) = \O_L^1\]
    for every closed point $y \in Y_{b_0}(\overline{\F}_p)$.
\end{theorem}

Note that $\O_L^1 = \{\pm 1\}$ unless either $L = \Q(\zeta_4)$ or $L = \Q(\zeta_6) = \Q(\sqrt{-3})$. In these exceptional cases, the automorphism group is generically given by the cyclic subgroup generated by $\zeta_4$ resp. $\zeta_6$. See Example \ref{example: exceptional cases} for further discussion.

In order to prove Theorem \ref{thmmain1}, we reduce it to a statement about the associated reduced Rapoport-Zink moduli space $\M$ via $p$-adic uniformization. It is a main feature of split unitary Shimura varieties that their Rapoport-Zink spaces parametrize $p$-divisible groups within a given basic isogeny class, but without any additional structure of PEL type, compare Section \ref{Section: RZ space}. Next, we reduce the question to the $a=1$-locus $\mathcal M^\circ$, an open and dense subscheme of $\M$ with a particularly nice and uniform description. It allows us not only to prove Theorem \ref{thmmain1}, but also to compute the generic automorphism group of the universal $p$-divisible group completely.

\begin{theorem}\label{thmmain3}
Let $g = \gcd(m_0,n_0)$, $n = \frac{n_0}{g}$ and $m= \frac{m_0}{g}$. Let $D$ be the division algebra over $\Q_p$ of invariant $\frac{m}{n}$ with maximal order $\O_D$ and a uniformizer $\pi$. Let
\[N = gm(n-m)-n+1.\]
    
Then there is an open and dense subscheme $Y\subseteq \mathcal M^{\circ}$ such that at every point of $Y$ defined over an algebraically closed field, the automorphism group of the universal $p$-divisible group is isomorphic to the group $\Gamma_{\gen}$, which in the case of $n \nmid g$ is defined as
\[\Gamma_{\gen} := \Z_p^\times + \pi^N \Mat_g(\O_D)\]
and otherwise is defined as
\[\Gamma_{\gen} := \Z_p^\times + \pi^{N-1}\O_D + \pi^N\Mat_g(\O_D).\]
\end{theorem}

Since the case of signature $(g,g)$ for $g \geq 2$ has already been treated in \cite{viehmann2026oort}, we assume throughout the paper that the signature $(m_0,n_0-m_0)$ is different from $(g,g)$ for $g\geq 2$. In particular, this entails that $n \geq 3$.

Let us briefly comment on the relation of our results with the notion of depth, a numerical invariant for the complexity of Rapoport-Zink spaces (or more generally affine Deligne-Lusztig varieties) which was introduced by Schremmer and the third named author. We refer to \cite[Definition 2.1]{schremmer2025affinedelignelusztig} for a formal definition.

In the setting of split unitary Shimura data as in Section \ref{Section: Shimura varieties}, the depth is given by
\begin{equation*}
    \operatorname{depth}(G,\mu) = \frac{m_0(n_0-m_0)}{n_0}.
\end{equation*}

In terms of this invariant, we see that Theorem \ref{thmmain1} holds whenever $\operatorname{depth}(G,\mu) \geq 2$ and in certain cases when $1 < \operatorname{depth}(G,\mu) < 2$. In the remaining cases, we expect the generic automorphism group to be larger than $\{\pm 1\}$. See Remark \ref{remark: remaining cases} for further discussion.

\subsection{Overview of the paper}
In Section~\ref{Section: Shimura varieties} we introduce the split unitary Shimura variety $\cS$ by specifying its PEL datum and the moduli problem it represents. We also recall the Newton stratification on $\cS$.
In Section~\ref{Section: RZ space}, we introduce the Rapoport-Zink space associated to the split unitary PEL setup, and use the $p$-adic uniformization theorem to reduce  Theorem~\ref{thmmain1} to a problem about the generic automorphism group of the universal $p$-divisible group. In doing so, we also explain how to use Dieudonné theory to make the geometry of the $a=1$ locus of the Rapoport-Zink space in question more explicit. Section~\ref{Section: automorphisms of lattices} is the main computational part of the paper and deals with the proof of Theorem~\ref{thmmain3}, i.e.~ with the generic automorphism group of the universal $p$-divisible group. In this section, we make crucial use of truncations of Dieudonné modules. Finally, in Section~\ref{Section: generic torsion} we study the torsion in the aforementioned generic automorphism group in order to deduce Theorem~\ref{thmmain1}.

\subsection{Acknowledgements}
We thank Chia-Fu Yu for reading a first version of this paper and pointing out a mistake in our original statement of Theorem \ref{thmmain1}. Furthermore we thank Pol van Hoften, Valentijn Karemaker and Ioannis Zachos for helpful conversations. During this work, the authors were partially supported by the DFG through the Collaborative Research Centre CRC 1442 ‘Geometry: Deformations and Rigidity’, under Germany's Excellence Strategy EXC 2044/2 – 390685587, ‘Mathematics M\"unster: Dynamics–Geometry–Structure’ and by the third named author's Leibniz prize.

\section{Split unitary Shimura varieties}\label{Section: Shimura varieties}
In this section, we recall the relevant split unitary Shimura varieties. Their characteristic feature is that the associated local reductive group at $p$ can be identified with a general linear group (at least up to a $\G_m$-factor). In our exposition, we follow \cite[Chapter 5]{haines2004svbadreduction}.

\subsection{Split unitary PEL datum}
Fix $\overline{\Q} \hookrightarrow \C \iso \overline{\Q}_p$. Let $L$ be an imaginary quadratic number field with conjugate complex embeddings $\nu,\nu^*: L \hookrightarrow \C$ (which we may assume to factor through our fixed embedding $\overline{\Q} \hookrightarrow \C$). We assume that the prime $p$ splits in $L$ as
\[(p) = \mathfrak p \mathfrak p^*,\]
where $\mathfrak p$ is determined by our fixed embedding $\overline{\Q} \hookrightarrow \overline{\Q}_p$, and $\mathfrak p^*\neq \mathfrak p$ is its conjugate under the non-trivial element of $\Gal(L/\Q)$. Thus we may identify $L_\mathfrak p = L_{\mathfrak p^*}= \Q_p$.

Further, we fix a central division algebra $\mathbf D$ over $L$ of dimension $n_0^2$, equipped with an involution $*$ which extends the non-trivial Galois action on $L$ over $\Q$. By our assumptions on $L$, we have an identification
\begin{equation*}
    \mathbf D \otimes \Q_p = \mathbf D_\mathfrak p \times \mathbf D_{\mathfrak p^*},
\end{equation*}
with both factors being central simple algebras over $\Q_p$. Furthermore, the involution induces an isomorphism $\mathbf D_\mathfrak p \overset{\sim}{\to} \mathbf D_{\mathfrak p^*}^{\opp}$. We assume that $\mathbf D$ splits at $p$, and fix an identification 
\begin{equation*}
    \mathbf D \otimes \Q_p = \mathbf D_\mathfrak p \times \mathbf D_{\mathfrak p^*} = \Mat_{n_0}(\Q_p) \times \Mat_{n_0}(\Q_p).
\end{equation*}
We may choose this identification so that the involution $*$ on $\mathbf D \otimes \Q_p$ gets identified with the involution $(X,Y) \mapsto (Y^T,X^T)$ of tuples of matrices.

We additionally suppose that we are given an $\R$-algebra homomorphism
\[h_0: \C \to \mathbf D \otimes_\Q \R\]
such that $h_0(z)^* = h_0(\overline{z})$ and such that the involution $x \mapsto h_0(i)^{-1}x^*h_0(i)$ is positive.

Let us now specify the (rational) PEL datum $(B,\iota,V,(\cdot,\cdot),h_0)$. We set $B= \mathbf D^{\opp}$ and view $V=\mathbf D$ as a left $B$-module. Note that we may interpret $h_0$ as an $\R$-algebra homomorphism $h_0: \C \to \End_B(V) \otimes \R$.

To define the involution $\iota$ on $B$, one first shows that there exists an element $\xi \in \mathbf D^\times$ satisfying $\xi^* = -\xi$. Let $x^\iota := \xi x^* \xi^{-1}$ for all $x \in B$ and consider the non-degenerate alternating pairing
\[(\cdot,\cdot): \mathbf D \times \mathbf D \to \Q,\quad (x,y) := \Tr_{\mathbf D/\Q}(x\xi x^*),\]
which satisfies $(bx,y) = (x,b^\iota y)$ and $(h_0(z)x,y) = (x,h_0(\overline{z})y)$. Upon possibly replacing $\xi$ with $-\xi$, we may also assume that $(\cdot,h_0(i)\cdot)$ is positive definite.

The associated reductive group $\mathbf G$ over $\Q$ is defined on $R$-points by
\[\mathbf G(R) = \{x \in (D \otimes_\Q R)^\times \mid x^*x \in R^\times\}.\]
By our assumptions on $L$ and $\mathbf D \otimes \Q_p$, the local reductive group $G = \mathbf G_{\Q_p}$ evaluated on $R$-points is
\[G(R) = \{(x_1,x_2) \in \GL_{n_0}(R) \times \GL_{n_0}(R) \mid x_2^* x_1 = c \text{ for some } c \in R^\times \}\]
and it thus may be identified with the split group $\GL_{n_0} \times~\G_{m}$ over $\Q_p$ by sending $(x_1,x_2)$ to $(x_1,c)$.

\begin{remark}\label{Remark: Hasse principle}
    By the discussion in \cite[§7]{kottwitz1992svfinitefields}, the reductive group $\mathbf G$ satisfies the Hasse principle. For even $n_0$, this is the case for any PEL datum of type (A) as considered in \cite{kottwitz1992svfinitefields}, but for odd $n_0$ this is due to the fact that $L$ is imaginary quadratic.
\end{remark}

The datum $(B,\iota,V,(\cdot,\cdot),h_0)$ defines a PEL datum of type (A) in the classification of Kottwitz. The complexification of $h_0$ can be written as
\[h_{0,\C}: \C \otimes_\R \C \simeq \C \times \C \to \mathbf D \otimes \C = \mathbf D \otimes_{F,\nu} \C \times \mathbf D \otimes_{F,\nu^*} \C,\]
where the first isomorphism is defined by $z_1 \otimes z_2 \mapsto (z_1z_2,\overline{z}_1z_2)$. The minuscule cocharacter associated with the PEL datum is then given by 
\[\mu: \C \to \mathbf D \otimes \C, \quad z \mapsto h_{0,\C}(z,1)\]
and we may decompose 
\[\mathbf D_\C = \mathbf D \otimes_{F,\nu} \C \times \mathbf D\otimes_{F,\nu^*} \C.\]
Choosing isomorphisms $\mathbf D \otimes_{F,\nu} \C \iso \mathbf D \otimes_{F,\nu^*} \C \iso \Mat_{n_0}(\C)$, we may write 
\[\mu(z) = \diag(1^{m_0},(z^{-1})^{n_0-m_0}) \times \diag((z^{-1})^{m_0},1^{n_0-m_0})\]
for some integer $m_0 \in \llbracket1,n_0-1\rrbracket$. We say that the split unitary PEL datum is of signature $(m_0,n_0-m_0)$, which is equivalent to having an isomorphism
\[\mathbf G_\R \iso \GU(m_0,n_0-m_0).\]
Our assumptions imply that the reflex field $\mathbf E$ of $\mu$ localizes to $\mathbf E_\mathfrak p = \Q_p$. 

Next we specify the integral data. For this, we fix an isomorphism $\mathbf{D}_{\Q_p} \cong \Mat_{n_0}(\Q_p) \times \Mat_{n_0}(\Q_p)$ such that the involution $\iota$ goes over to $(X,Y) \mapsto (Y^T,X^T)$, and let $\O_B \subset B$ be the unique maximal order of $B$ such that $\O_B \otimes \Z_p$ gets identified with $\Mat_{n_0}^{\opp}(\Z_p) \times \Mat_{n_0}^{\opp}(\Z_p)$. We also fix an $\O_B$-lattice $\mathbf \Lambda \subset V$ such that
\[\mathbf\Lambda \otimes \Z_p = \Mat_{n_0}(\Z_p) \oplus \chi^{-1}\Mat_{n_0}(\Z_p),\]
which is a self-dual $\O_B \otimes \Z_p$-lattice. Here we identified $\xi \in \mathbf{D}^\times$ with an element of the form $(\chi^T,-\chi)$ for some $\chi \in \GL_{n_0}(\Q_p)$.

\subsection{Moduli problem}
Fix a sufficiently small compact open subgroup $K^p \subset \mathbf{G}(\A_f^p)$ that stabilizes $\mathbf \Lambda \otimes \widehat{\Z}^{(p)}$. We introduce the moduli problem in terms of isomorphisms  of abelian varieties instead of prime-to-$p$ quasi-isogenies, so that the automorphism groups we want to consider are well-defined.
\begin{definition}[\cite{lan2013arithmetic}, Def. 1.4.1.4]
    Let $\Sh_{K^p}$ denote the moduli stack in schemes over  $\O_{\mathbf{E}} \otimes \Z_{(p)}$ that sends $S$ to the groupoid whose objects consist of quadruples $(A,i,\lambda,\overline{\eta})$, where
    \begin{itemize}
        \item $A$ is an abelian scheme over $S$. 
        \item $i: \O_B \hookrightarrow \End_S(A)$ is an $\O_B$-action that satisfies the Kottwitz condition
        \begin{equation*}
            \det(i(b) \mid \Lie A) = \det (b \mid V_1)
        \end{equation*}
        for all $b \in \O_B$. 
        \item $\lambda: A \to A^\vee$ is a prime-to-$p$ polarization such that the induced Rosati involution on $\End_S(A)$ is compatible with the involution on $\O_B$ via the $\O_B$-action $i$.
        \item $\overline{\eta}$ is a $K^p$-level structure, see \cite[Def. 1.3.7.6]{lan2013arithmetic}.
    \end{itemize}
    An isomorphim of quadruples is an isomorphism that respects the $\O_B$-action, the polarization (on the nose and not up to a similitude factor) and the $K^p$-level structure.
\end{definition}

For sufficiently small $K^p$, the above moduli problem is representable by a smooth, quasi-projective scheme over $\O_{\mathbf E} \otimes \Z_{(p)}$, which is the integral model of the Shimura variety associated to the above PEL datum. From now on, we disregard the $K^p$-level structure from our notation and denote the base change to $\Spec \F_p$ by $\cS$.

Explicitly, let $(A,i,\lambda)$ denote the universal abelian scheme with $\mathbf G$-structure over $\cS$. For any $\Spec \F_p$-scheme $T$ let $x \in \cS(T)$ be an arbitrary $T$-point of $\cS$. Then the automorphism group of $(A,i,\lambda)$ at $x$ is defined as
    \[\Aut(A_x,i_x,\lambda_x) = \{\varphi: A_x \overset{\sim}{\to} A_x \mid \varphi \circ i_x(b) = i_x(b) \circ \varphi \text{ for all } b \in \O_B \text{ and } \varphi^\vee \circ \lambda_x \circ \varphi = \lambda_x \}.\]
If $x$ is a geometric point of $\cS$, then it is a classical fact that $\Aut(A_x,i_x,\lambda_x)$ is finite.

Next, we recall the Newton stratification on $\cS$. Consider the local group $G = \mathbf G_{\Q_p}$. Let $\cH = (A[p^\infty],i[p^\infty],\lambda[p^\infty])$ be the universal $p$-divisible group with $G$-structure over $\cS$.  Recall that the Kottwitz set $B(G)$ is defined as the set of $\sigma$-conjugacy classes in $G(\breve{\Q}_p)$. It classifies $p$-divisible groups with $G$-structure up to isogeny, or equivalently isocrystals with $G$-structure up to isomorphism, both over arbitrary algebraically closed fields. Our fixed identification $G = \GL_{n_0,\Q_p} \times~\G_{m,\Q_p}$ induces an identification
\begin{equation*}
    B(G) = B(\GL_{n_0}) \times B(\G_m) = B(\GL_n) \times \Z.
\end{equation*}
For $b \in B(G)$, the subset $\cS_b \subset \cS$ that consists of those geometric points $x$ of $\cS$ such that the $p$-divisible group with $G$-structure $\cH_x$ belongs to the isogeny class $b \in B(G)$, is locally closed. We endow $\cS_b$ with its reduced subscheme structure, and denote the resulting locally closed subscheme of $\cS$ also by $\cS_b$. These locally closed subschemes are called the Newton strata of $\cS$. 

The non-empty Newton strata are indexed by the subset of neutrally acceptable elements $B(G,\mu)$, and by our identification
\begin{equation*}
    B(G,\mu) = B(\GL_{n_0},(0^{n_0-m_0},(-1)^{m_0})) \times B(\G_m,0) = B(\GL_{n_0},(0^{n_0-m_0},(-1)^{m_0})).
\end{equation*}
Thus we have an injective map $B(G,\mu) \hookrightarrow B(\GL_{n_0})$, so we may think of its elements as  classical Newton polygons that lie below the polygon associated to the cocharacter $(1^{m_0},0^{n_0-m_0})$.

There is a unique closed Newton stratum in $\cS$, called the basic stratum. It is labelled by the element $b_0 \in B(G,\mu)$ corresponding to the Newton polygon of constant slope $\frac{m_0}{n_0}$. 

\section{The Rapoport-Zink space}\label{Section: RZ space}
In this section, we reduce Theorem~\ref{thmmain1} to a problem about the geometry of the associated Rapoport-Zink space via $p$-adic uniformization. We keep the notation established in the last section. 

Given a point $x \in \cS_{b_0}(\overline{\F}_p)$, it is well-known that we have an injection
\[\Aut(A_x,i_x,\lambda_x) \hookrightarrow \Aut(\cH_x).\]
The right-hand side is significantly larger, but also much easier to work with due to Dieudonné theory. We compute it generically in Section~\ref{Section: automorphisms of lattices} (cf. Theorem~\ref{Theorem: generic automorphism group}). Note that the above injection takes image in the torsion subgroup of $\Aut(\cH_x)$. Thus we study the generic torsion subgroup in Section~\ref{Section: generic torsion} (cf. Proposition \ref{Proposition: generic torsion in p-div automorphism}), from which we deduce Theorem~\ref{thmmain1}. For now, let us recall how to relate the geometry of the Shimura variety to the geometry of a Rapoport-Zink space.

The splitting conditions on our PEL datum at $p$ have strong implications for any $p$-divisible group $(X,i,\lambda)$ with $G$-structure over an arbitrary base. Recall that this means that $X$ is equipped with an endomorphism structure by 
\[\O_B \otimes \Z_p = \Mat_{n_0}(\Z_p) \times \Mat_{n_0}(\Z_p)\]
and with a polarization $\lambda: X \overset{\sim}{\to} X^\vee$. Letting $(\id_{n_0},0)$ resp. $(0,\id_{n_0})$ act on $X$, we see that $X$ splits into two $p$-divisible groups
\[X = X_1 \oplus X_2,\]
each of height $n_0$. Compatibility between polarization and endomorphism structure imply that $X_2 \cong X_1^\vee$ and that 
\[\lambda = \swap: X_1 \oplus X_2 \to X_2 \oplus X_1\]
is the swap map (at least up to a constant). Thus giving $X$ is equivalent to giving $X_1$, which is a $p$-divisible group of height $n_0$ without any additional PEL type structure.

Furthermore, automorphisms of $(X,i,\lambda)$ are also fully determined by automorphisms of $X_1$.
\begin{lemma}\label{Lemma: comparison automorphism groups}
    Let $(X,i,\lambda)$ be a $p$-divisible group with $G$-structure and decompose it into $X = X_1 \oplus X_2$ as before. Then
    \[\Aut(X_1) \overset{\sim}{\longrightarrow} \Aut(X,i,\lambda) \text{ via } f \mapsto \begin{pmatrix}
        f & 0 \\
        0 & (f^{-1})^\vee
    \end{pmatrix}.\]
    In particular, the identification preserves torsion automorphisms. Furthermore, the subgroup $\Z_p^\times \subset \Aut(X_1)$ is identified with
    \[\{(a,a^{-1})\mid a \in \Z_p^\times\} \subset \Aut(X,i,\lambda)\].
\end{lemma}
\begin{proof}
    It is clear that the map is well-defined and injective. To prove surjectivity, note that any $g \in \Aut(X,i,\lambda)$ commutes with the action by $(\id_{n_0},0)$ and $(0,\id_{n_0})$, so it preserves both $X_1$ and $X_2$. Thus $g$ can be represented by a diagonal matrix of the form
    \[g=\begin{pmatrix}
        g_1 & 0 \\
        0 & g_2
    \end{pmatrix}\]
    with automorphisms $g_i\in \Aut(X_i)$ for $i =1,2$. Compatibility with the polarization implies that 
    \[\begin{pmatrix}
        0 & 1 \\
        1 & 0
    \end{pmatrix} = \begin{pmatrix}
        g_1 & 0 \\
        0 & g_2
    \end{pmatrix}^\vee \begin{pmatrix}
        0 & 1 \\
        1 & 0
    \end{pmatrix} \begin{pmatrix}
        g_1 & 0 \\
        0 & g_2
    \end{pmatrix} = \begin{pmatrix}
        0 & 1 \\
        1 & 0
    \end{pmatrix} \begin{pmatrix}
        g_2^\vee g_1 & 0 \\
        0 & g_1^\vee g_2
    \end{pmatrix},\]
    which implies that $g_2 = (g_1^{-1})^\vee$.
\end{proof}

Thus for all our purposes, it suffices to introduce Rapoport-Zink spaces which parametrize $p$-divisible groups of height $n_0$ without any additional PEL type structure, i.e.~Rapoport-Zink spaces for $\GL_{n_0}$.

As a framing object, we fix a $p$-divisible group $\X$ of height $n_0$ in the isogeny class of $b_0$, which we may assume to be defined over $\F_p$ and also to be minimal (see \cite{oort2005minimal}). Let $\breve{\Z}_p = W(\overline{\F}_p)$, and consider the functor $\M$ on the category of schemes over $\breve{\Z}_p$ where $p$ is locally nilpotent, that assigns to such a scheme $S$ the set of isomorphism classes of tuples $(X,\rho)$, where
\begin{itemize}
    \item $X$ is a $p$-divisible group of height $n_0$ over $S$
    \item $\rho: \X_{\overline{S}} \to X_{\overline{S}}$ is a quasi-isogeny over the reduction of $S$ mod $p$.
\end{itemize}
By \cite[Theorem 3.25]{rapoport1996periodspaces}, this functor is representable by a formal scheme locally formally of finite type over $\Spf \breve{\Z}_p$, called the Rapoport-Zink space associated to the basic isogeny class $b_0$. In the sequel, we are most interested in the underlying reduced scheme over $\overline{\F}_p$, which we from now on also denote by $\M$. Note that the group $J_{b_0}(\Q_p)$ of self-quasi-isogenies of the framing object $\X$ clearly acts on the Rapoport-Zink space $\M$.

\begin{remark}\label{Remark: RZ space of G vs. GL_n}
    If we denote by $\M_G$ the \textit{actual} Rapoport-Zink space associated with the group $G$ and the basic isogeny class $b_0$, then our identification $G = \GL_{n_0} \times \G_m$ implies that
    \[\M_G = \coprod_{\Z} \M\]
    because the Rapoport-Zink space associated with $\G_m$ consists of $\Z$-many copies of a single point. We will include the difference in our notation, but it is clear that it suffices to analyze generic conditions on $\M$ instead of $\M_G$.
\end{remark}

\subsection{\texorpdfstring{$p$-adic uniformization}{p-adic uniformization}}\label{subsection: p-adic uniformization}
By the $p$-adic uniformization theorem of Rapoport-Zink \cite[Theorem 6.30]{rapoport1996periodspaces}, there is an isomorphism of schemes over $\overline{\F}_p$
\[\Theta: I(\Q) \backslash \M_G \times \mathbf G(\A_f^p)/K^p \overset{\sim}{\longrightarrow} \cS_{b_0},\]
where $I(\Q)$ is the group of self-quasi-isogenies of a basepoint $(A,i,\lambda) \in \cS_{b_0}(\overline{\F}_p)$ whose $p$-divisible group is identified with $\X \oplus \X^\vee$. Here we implicitly use that the Hasse principle is satisfied for $\mathbf G$, see Remark \ref{Remark: Hasse principle}. In particular, the morphism
\[\M_G \times \mathbf G(\A_f^p)/K^p \to I(\Q) \backslash \M_G \times \mathbf G(\A_f^p)/K^p \overset{\Theta}{\to} \cS_{b_0}\]
is surjective and open, thus sends dense open subsets to dense open subsets. 

Thus in order to prove Theorem \ref{thmmain1}, the most important step is to compute the torsion in the automorphism group of the universal $p$-divisible group restricted to a generic subspace of $\M_G$, or equivalently $\M$, in order to then transport this generic subspace to $\cS_{b_0}$ via the $p$-adic uniformization morphism. 

In this direction, the next reduction is obtained by considering Oort's $a$-invariant.

\begin{definition}
    The $a$-number of a $p$-divisible group $X$ over an algebraically closed field $k$ of characteristic $p$ is defined as $a(X) = \dim_k\Hom(\alpha_p,X)$. 
\end{definition}

Let $\M^\circ \subset \M$ be the locus over which the $a$-number of the universal $p$-divisible group is equal to $1$. By \cite[Lemma 4.7]{viehmann2008modulispaces}, building on work of Oort, $\M^\circ$ is dense and open in $\M$. Thus we reduce to studying the generic torsion in the automorphism group of the $p$-divisible group over $\M^\circ$, which is more amenable to computations as we shall see shortly.

\subsection{Dieudonné modules}
In this subsection, we make the irreducible components of $\M^\circ$ and the group $J_{b_0}(\Q_p)$ more explicit by means of Dieudonné theory.

Recall the equivalence between $p$-divisible groups of height $n_0$ over perfect fields $k$ of characteristic $p$ and Dieudonné modules, i.e.~finite free $W(k)$-modules $M$ of rank $n_0$ equipped with a $\sigma$-linear Frobenius endomorphism $F: M \to M$ that satisfies $pM \subset FM$. Here $\sigma$ denotes the induced Frobenius on $W(k)$. Considering the $p$-divisible group only up to isogeny corresponds to tensoring $M$ with $W(k)[\frac{1}{p}]$. This leads to the notion of an isocrystal. Inside the isocrystal, we may define Verschiebung as $V = pF^{-1}$. Note that any Dieudonné module is stable under $V$.

Since our framing object $\X$ is defined over $\F_p$, its associated isocrystal $(N,F)$ is defined over $\Q_p$. Denote the base change of $N$ to $\breve{\Q}_p$ by $\breve{N}$. Then by the Dieudonné equivalence, points $(X,\rho) \in \M(\overline{\F}_p)$ correspond bijectively to Dieudonné lattices $(M,F) \subset (\breve N,F)$, i.e. $M \subset \breve N$ is a free $\breve{\Z}_p$-submodule of $N$ of rank $n_0$ satisfying $M \supset F(M) \supset pM$. We have
\[\Aut(X) = \Aut(M,F) = \{\gamma: M \to M \mid \gamma  \text{ linear and } F\gamma = \gamma F\}.\]

The isocrystal $\breve{N}$ over $\breve{\Q}_p$ can be further decomposed according to the Dieudonné-Manin classification. For this, let $g = \gcd(m_0,n_0)$ and let $n= \frac{n_0}{g}$ and $m = \frac{m_0}{g}$. Then $\breve{N}$ splits as the direct sum of simple isocrystals
\[\breve N = \bigoplus_{i=1}^g \breve N_i,\]
where each $\breve{N_i}$ is simple of slope $\frac{m}{n}$. Furthermore, for each $i \in \llbracket 1,g\rrbracket$, we can fix a vector $e_{i,0} \in \breve N_i$ that  satisfies $F^ne_{i,0} = p^me_{i,0}$. 

We consider the following two operators on $\breve{N}$.
\begin{definition}
    Let $a,b \in \Z$ such that $an+bm=1$. Let
    \begin{itemize}
        \item $\tau_0:= p^{-m}F^n: \breve{N} \to \breve{N}$. It is $\sigma^n$-linear.
        \item $\tau_1:= p^aF^b: \breve{N} \to \breve{N}$. It is $\sigma^b$-linear.
    \end{itemize}
\end{definition}

For any $l \geq 1$, we let $e_{i,l}= \tau_1^le_{i,0} \in \breve{N}_i$. Then the vectors $e_{i,l}$ for $i \in \llbracket1,g\rrbracket$ and $l \in \llbracket 0,n-1\rrbracket$ form a $\breve \Q_p$-basis of $\breve{N}$. One checks immediately that $e_{i,l+n}=pe_{i,l}$, $F(e_{i,l})=e_{i,l+m}$ and $V(e_{i,l})= e_{i,l+n-m}$, so clearly this basis is particularly suited for computations. With respect to this basis, $\tau_0$ identifies with $n$-th iterated Frobenius.

\begin{definition}
    Let $\Lambda_0 \subset \breve{N}$ be the lattice generated by the vectors $e_{i,l}$ for $i \in \llbracket 1,g\rrbracket$ and $l \geq 0$. 
\end{definition}

Then $\Lambda_0$ is a Dieudonn\'e lattice. Also by the Dieudonné-Manin classification, we can identify the endomorphism algebra of each $\breve N_i$ with the unique division algebra $D$ over $\Q_p$ of invariant $\frac{m}{n}$. We can uniquely extend the valuation on $\Q_p$ to a valuation $\operatorname{val}: D^\times \to \frac{1}{n}\Z$, which induces the ring of integers $\O_D$. Define the $\breve{\Q}_p$-linear automorphism $\pi: \breve{N}_i \to \breve N_i$ by setting $\pi(e_{i,l})= e_{i,l+1}$, then $\pi$ is an automorphism of isocrystals. Because it satisfies $\pi^n=p$, we see that $\pi$ is a uniformizer of $\O_D$ with valuation $\operatorname{val}(\pi) = \frac{1}{n}$. With respect to the chosen basis, we have identifications
\begin{equation}\label{eq identification J_b}
J_{b_0}(\Q_p) \simeq \Aut(\breve N,F) \simeq \GL_g(D)
\end{equation}
which we assume to be fixed from now on. Equation~\eqref{eq identification J_b} restricts to an identification
\begin{equation*}
    \Aut(\Lambda_0,F) \simeq \GL_g(\O_D).
\end{equation*}

\begin{remark}\label{remark: scalars in Aut N}
    Note that $\Aut(\breve{N},F)\subset \Aut_{\breve{\Q}_p}(\breve{N})$, which contains multiplication by scalars in $\breve{\Q}_p^\times$. However as $F$ is $\sigma$-linear, in $\Aut_{\breve{\Q}_p}(\breve{N})$ we have $\Aut(\breve N,F)\cap \breve{\Q}_p^\times=\Q_p^\times$. Therefore through identification \eqref{eq identification J_b}, only elements of $\Q_p^\times\subset \GL_g(D)$ act as genuine scalar multiplication on $\breve{N}$, even though $D$ contains an unramified subfield isomorphic to $\Q_{p^n}$.
\end{remark}

Let us now turn to the locus $\M^\circ$ where the $a$-number is equal to $1$. By Dieudonné theory, we can reinterpret the $a$-number of a $p$-divisible group $X$ in terms of its Dieudonné module $M$, namely 
\[a(X) = a(M) := \dim_k M/(FM+VM)\]
which we call the $a$-number of $M$. Let $\breve{\Z}_p\{F,V\}$ denote the non-commutative polynomial ring in variables $F$ and $V$ over $\breve{\Z}_p$. We define the Dieudonné ring for $\overline{\F}_p$ to be
\[\mathcal D := \breve{\Z}_p\{F,V\}/(a^\sigma F-Fa, aV-Va^\sigma, FV-p,VF-p),\]
where we use the notation $a^\sigma := \sigma(a)$. Then for any Dieudonné module $M$, $a(M) =1$ if and only if there exists some $v \in M$ such that $M = \mathcal D v$.

Our next goal is to describe how to decompose $\M^\circ$ into irreducible components, or equivalently into connected components. 

\begin{definition}
    A Dieudonné lattice $M \subset N$ is called a vertex lattice if it is stable under $\tau_0$ and $\tau_1$.
\end{definition}
For example, the Dieudonné lattice $\Lambda_0$ defined above is a vertex lattice. 
\begin{remark}
To define $\tau_1$, we fixed integers $a,b$ with $an+bm=1$. The notion of vertex lattice as well as the vertex lattice associated with a given $M$ that we study below do not depend on this choice. Indeed, if $\tau_1'$ is the operator defined using a possibly different choice $a'$, $b'$, then $(a-a')n+(b-b')m=0$, in other words there is an $l$ with $a-a'=lm$ and $b-b'=ln$. Hence $\tau_1$ and $\tau_1'$ differ by a power of the (well-defined) operator $\tau_0$.
\end{remark}

The following is a variant of Zink's lemma, which is an important step in the classification of isocrystals. Using this classification, we can easily deduce the variant we need.
\begin{lemma}
    Let $M \subset \breve{N}$ be a Dieudonné lattice. Then $M_\tau = \sum_{i,j \geq 0} \tau_0^i\tau_1^jM$ is the unique smallest vertex lattice in $\breve{N}$ that contains $M$. In particular, there is a natural injection 
    \[\Aut(M,F) \hookrightarrow \Aut(M_\tau,F).\]
\end{lemma}
\begin{proof}
    Let $c \in \Z$ such that $M \subseteq p^c\Lambda_0$, then because $\Lambda_0$ is stable under $\tau_0$ and $\tau_1$ we see that $M \subseteq M_\tau \subseteq p^c\Lambda_0$. Thus $M_\tau$ is again a lattice. The rest of the assertion is then clear.
\end{proof}

\begin{proposition}
    Let $S$ be the set of vertex lattices, regarded as a discrete set. Then
    \begin{enumerate}
        \item $S$ is the $J_{b_0}(\Q_p)$-orbit of $\Lambda_0$.
        \item The map $\varphi: \M^\circ(\overline{\F}_p) \to S$ sending $M$ to $M_\tau$ defines a locally constant and surjective morphism of schemes $\varphi:\M^\circ \to S$.
        \item For each $s \in S$, the inverse image $\varphi^{-1}(s)$ is an irreducible component of $\M^\circ$.
    \end{enumerate}
\end{proposition}
\begin{proof}
    The first assertion is \cite[Lemma 4.10]{viehmann2008modulispaces} and the other assertions follow from \cite[Section 4.4]{viehmann2008modulispaces}.
\end{proof}
In this way, vertex lattices parametrize the irreducible components of $\M^\circ$ and $J_{b_0}(\Q_p)$ acts transitively on the set of irreducible components. We denote the irreducible component of $\M^\circ$ labelled by the vertex lattice $\Lambda_0$ by $\mathcal C_{\Lambda_0}$. 

Clearly the action by $J_{b_0}(\Q_p)$ does not change the automorphism group of the $p$-divisible group, at least up to isomorphism. Thus in order to prove Theorems \ref{thmmain1} and \ref{thmmain3}, we are left to study the generic automorphism group of the universal $p$-divisible group over $\mathcal C_{\Lambda_0}$. Put differently, we are going to do explicit calculations with Dieudonné lattices $M \subset \breve{N}$ that satisfy $a(M)=1$ and $M_\tau = \Lambda_0$, in particular this implies that
\[\Aut(M,F) \hookrightarrow \Aut(\Lambda_0,F) = \GL_g(\O_D).\] The aim is to identify the generic automorphism group as a subgroup of $\GL_g(\O_D)$.

\subsection{Explicit coordinates for \texorpdfstring{$\mathcal C_{\Lambda_0}$}{C(Lambda0)}}
In this subsection, we describe explicit coordinates for the irreducible component $\mathcal C_{\Lambda_0}$.

Every element $v \in \Lambda_0$ admits a unique decomposition $v= \sum_{i=1}^g \sum_{l \geq0} [a_{i,l}] e_{i,l}$, where $a_{i,l}\in \overline{\F}_p$, where $[\cdot]$ denotes the Teichmüller lift, and where the sum converges in the $p$-adic topology. The following lemma is a characterization of vectors $v \in \Lambda_0$ that generate a Dieudonné lattice $\mathcal D \cdot v\in \mathcal C_{\Lambda_0}(\overline{\F}_p)$.

\begin{lemma}\label{Lemma: a=1 dieudonne lattices generators}
Let $v\in \Lambda_0$ as above. The following conditions are equivalent:
\begin{enumerate}
    \item The $\breve{\Z}_p$-module $\D\cdot v$ is an element of $\mathcal C_{\Lambda_0}(\overline{\F}_p)$.
    \item The elements $a_{1,0},...,a_{g,0} \in \overline{\F}_p$ are linearly independent over $\F_{p^n}$.
    \item For some $n_0\in \Z$, the matrix $(a_{i,0}^{\sigma^{n_0+jn}})_{i,j\in\llbracket1,g\rrbracket}$ is invertible.
    \item For any $n_0\in \Z$, the matrix $(a_{i,0}^{\sigma^{n_0+jn}})_{i,j\in\llbracket1,g\rrbracket}$ is invertible.
\end{enumerate}
\end{lemma}
\begin{proof}
    The first equivalence is proven in \cite[Lemma 4.8]{viehmann2008modulispaces}. The other equivalences follow from Dedekind's theorem on the independence of characters (see \cite[Corollary 5.16]{milne2022fields}).
\end{proof}

Let $\Lambda_\bullet$ be the filtration of $\Lambda_0$ defined by $\Lambda_j = \pi^j\Lambda_0$ for every $j \geq 0$, and for any Dieudonné lattice $M \in \mathcal C_{\Lambda_0}(\overline{\F}_p)$ we let $M_\bullet = M \cap \Lambda_\bullet$. For $s \in \Z_{\geq 1}$, we consider the truncation map
\[[\cdot]_s: \Lambda_0 \to \Lambda_0/\Lambda_s.\]
Let us also define $\mathcal B_j = \{(a,b) \in \Z_{\geq0}^2 \mid am+b(n-m)=j\}$ for $j \geq 0$. We may enumerate the elements of $\mathcal B_j$ as $(a_1,b_1),...,(a_{g'},b_{g'})$ such that $a_1 < a_2 < \cdots < a_{g'}$, where $g' = |\mathcal B_j|$. Then $(a_l,b_l) = (a_1+(l-1)(n-m),b_1-(l-1)m)$ and 
\begin{align}\label{eq differences in B_j enumerated}
    a_l-b_l = a_1-b_1+(l-1)n.
\end{align}
for all $l \in \llbracket1,g'\rrbracket$.

\begin{lemma}\label{Lemma: basis of level s}
    Let $v\in \Lambda_0$ be as in Lemma \ref{Lemma: a=1 dieudonne lattices generators} and let $M= \mathcal D\cdot v \in \mathcal C_{\Lambda_0}(\overline{\F}_p)$. Then for all $j \geq 0$, the $\overline{\F}_p$-vector space $M_{j}/M_{j+1}$ is generated by the family of vectors $\{[F^aV^bv]_{j+1}\}_{(a,b)\in \mathcal B_j}$. If $j < gm(n-m)$, this family is in fact a basis.
\end{lemma}
\begin{proof}
    For any $s \geq 1$, $\ker([\cdot]_s)$ is generated by the vectors $\{e_{i,l} \mid i \in \llbracket 1,g\rrbracket, l \geq s\}$ and $[\cdot]_s$ maps $M$ to $M/M_s$. Also, $\Lambda_j/\Lambda_{j+1}$ is an $\overline{\F}_p$-vector space with basis $([e_{1,j}]_{j+1},...,[e_{g,j}]_{j+1})$.

    Since $F^aV^be_{i,l} = e_{i,l+am+b(n-m)}$, we see that for any $s \geq 1$.
    \begin{equation*}
    a+b \geq s \implies F^aV^b(\Lambda_0)\subset \ker([\cdot]_s).\end{equation*}
    This implies that $M/M_1$ is freely generated over $\overline{\F}_p$ by $[v]_1$.

    We first prove the statement via induction for $j \in \llbracket 1,gm(n-m)\rrbracket$, with induction hypothesis that for all $j'\in \llbracket 0,j-1\rrbracket$ the $\overline{\F}_p$-vector space $M_{j'}/M_{j'+1}$ is freely generated by the vectors $[F^aV^bv]_{j'+1}$.

    Any $u \in M_j \setminus M_{j+1}$ admits a decomposition
    \begin{equation}\label{eq elements of Mj}
        u = \sum_{a,b\geq 0} u_{a,b}F^aV^bv = \sum_{i=1}^g \sum_{a,b,l\geq 0} u_{a,b} [a_{i,l}]^{\sigma^{a-b}}e_{i,l+am+b(n-m)}
    \end{equation}
    with $u_{a,b}\in\breve{\Z}_p$ for all $a,b\geq 0$. Let $s = \min \{am+b(n-m) \mid u_{ab} \neq 0\}$, then $u \in M_s$ and $s \leq j$. By induction hypothesis, we have $s=j$. Thus truncating~\eqref{eq elements of Mj} at level $j+1$, we get
    \[[u]_{j+1} = \sum_{(a,b) \in \mathcal B_j} \overline{u}_{a,b}[F^aV^bv]_{j+1},\]
    which proves that the vectors $[F^aV^bv]_{j+1}$ generate $M_j/M_{j+1}$. These elements are linearly independent if $j<gm(n-m)$. Indeed, we enumerate the elements of $\mathcal B_j$ as before and deduce from \eqref{eq differences in B_j enumerated} that
    \[[F^{a_k}V^{b_k}v]_{j+1}=\sum_{i=1}^g a_{i,0}^{\sigma^{a_k-b_k}}[e_{i,j}]_{j+1} = \sum_{i=1}^g a_{i,0}^{\sigma^{a_1-b_1+(k-1)n}} [e_{i,j}]_{j+1}.\]
    Lemma~\ref{Lemma: a=1 dieudonne lattices generators} implies that these vectors are linearly independent over $\overline{\F}_p$ if and only if $|\mathcal B_j|\leq g$, which is the case whenever $j<gm(n-m)$. 

    Using Lemma~\ref{Lemma: a=1 dieudonne lattices generators}, one also sees immediately that for $j \geq gm(n-m)$, the vectors $\{[F^aV^bv]_{j+1}\}_{(a,b) \in \mathcal B_j}$ generate $M_j/M_{j+1}$, which actually equals $\Lambda_j/\Lambda_{j+1}$ due to $|\mathcal B_j| \geq g$.
\end{proof}

\begin{lemma}\label{Lemma value PiN subset M}
Let $N = gm(n-m)-n+1$. Then $\Lambda_N \subset M$ is the largest vertex lattice contained in $M$. In particular, $N$ is the smallest integer satisfying $\Lambda_N \subset M$.
\end{lemma}
\begin{proof}
    We claim that for all $i \in \llbracket 1,n-1\rrbracket$ we have $|\mathcal B_{gm(n-m)-i}| = g$. Indeed, by Bezout's identity there are $c_i,d_i \in \Z$ such that
        \[c_im+d_i(n-m) = -i.\]
    Replacing $c_i$ by $c_i-\varepsilon(n-m)$ and $d_i$ by $d_i+\varepsilon m$ for suitable $\varepsilon \in \Z$, we may assume that $-(n-m) \leq c_i < 0$ and $0 \leq d_i \leq m$. Then $\mathcal B_{gm(n-m)-i}$ contains precisely the $g$-many tuples 
    \[((g-k)(n-m)+c_i, km+d_i) \in \mathcal B_{gm(n-m)-i} \text{ for } k \in \llbracket 0,g-1\rrbracket.\]

    Thus $\dim_{\overline{\F}_p} M_j/M_{j+1} = g$ is maximal for all $j \in \llbracket N,N+n-1\rrbracket$, by Lemma \ref{Lemma: basis of level s}. From $F(\Lambda_j) \subset \Lambda_{j+m}$ and $V(\Lambda_j) \subset \Lambda_{j+n-m}$, we conclude that $\dim_{\overline{\F}_p} M_j/M_{j+1}=g$ for all $j \geq N$. Since $M$ is a lattice, there exists an $\tilde{N} \gg 0$ such that $\Lambda_{\tilde N} \subset M$, or equivalently $\Lambda_{\tilde{N}} = M_{\tilde{N}}$. We now argue by induction, assume $\Lambda_n = M_n$ holds for some $n \in \llbracket N+1,\tilde{N}\rrbracket$ and consider the diagram
    \[\begin{tikzcd}
	0 & {M_{n}} & {M_{n-1}} & {M_{n-1}/M_n} & 0 \\
	0 & {\Lambda_n} & { \Lambda_{n-1}} & {\Lambda_{n-1}/\Lambda_n} & 0
	\arrow[from=1-1, to=1-2]
	\arrow[from=1-2, to=1-3]
	\arrow["{=}"', from=1-2, to=2-2]
	\arrow[from=1-3, to=1-4]
	\arrow[hook', from=1-3, to=2-3]
	\arrow[from=1-4, to=1-5]
	\arrow[hook', from=1-4, to=2-4]
	\arrow[from=2-1, to=2-2]
	\arrow[from=2-2, to=2-3]
	\arrow[from=2-3, to=2-4]
	\arrow[from=2-4, to=2-5]
    \end{tikzcd}\]
    By dimension count the right-hand arrow is an equality, thus by the 5-lemma $\Lambda_{n-1}=M_{n-1}$. Thus inductively $\Lambda_N = M_N$, and so $\Lambda_N \subset M$.

    Clearly $\Lambda_N$ is a vertex lattice, and in fact it is the largest vertex lattice contained in $M$. Indeed, we claim that $(M_j/M_{j+1}) \cap \F_{p^n}^g = \{0\}$ inside $\Lambda_j/\Lambda_{j+1} \cong \overline{\F}_{p}^g$ for any $j \geq0$ such that $|\mathcal B_j| < g$. Assuming this claim, let $M' \subset  M$ be a vertex lattice and set $M'_j = M' \cap \Lambda_j$. Since $M'$ is $\tau_1$-stable, the function $j \mapsto \dim_{\overline{\F}_p} M'_j/M'_{j+1}$ is non-decreasing. Furthermore, by our claim $\dim_{\overline{\F}_p} M_j'/M'_{j+1} =0$ for all $j \geq 0$ such that $|\mathcal B_j| < g$, because $M'$ is $\tau_0$-stable and thus generated by $\tau_0$-stable elements. Together this shows that $M'_j/M'_{j+1}=0$ for all $j \leq gm(n-m)-n$, which implies $M' \subset \Lambda_N$ by Lemma \ref{Lemma: basis of level s}.

    Finally, the claim follows from Lemma \ref{Lemma: independence over F_p^n} below by enumerating $\mathcal B_j$ as before as $(a_k,b_k) \in \mathcal B_j$, setting $x_i = a_{i,0}^{\sigma^{a_1-b_1}}$ and using \eqref{eq differences in B_j enumerated}.
\end{proof}

\begin{lemma}[cf.~\cite{viehmann2026oort}, Lemma 4.5]\label{Lemma: independence over F_p^n}
    Let $x_1,...,x_g \in \overline{\F}_p$ be linearly independent over $\F_{p^n}$, and let $\varepsilon_1,...,\varepsilon_g \in \F_{p^n}$. Let $k \leq g-2$ and consider the $g \times (k+2)$-matrix

\begin{align}\label{matrix: F_p^n minors}
    \begin{pmatrix}
        x_1 & x_1^{\sigma^n} & \cdots & x_1^{\sigma^{kn}} &\varepsilon_1 \\
        x_{2} & x_2^{\sigma^{n}} & \cdots & x_2^{\sigma^{kn}} &\varepsilon_2 \\
        \cdots & \cdots & \cdots & \cdots & \cdots \\
        x_g & x_g^{\sigma^{n}} & \cdots & x_g^{\sigma^{kn}} &\varepsilon_g 
    \end{pmatrix}
\end{align}
Then the rank of this matrix is not maximal if and only if $\varepsilon_i =0$ for all $i=1,\dots,g$.
\end{lemma}
\begin{proof}
    Obviously all maximal minors vanish if all $\varepsilon_i$ are zero. Thus assume there is an index $i \in \llbracket 1,g\rrbracket$ such that $\varepsilon_i \neq 0$, and we want to show that some maximal minor does not vanish. 
    
    Upon reordering we may assume that $i =1$, and upon considering the maximal minor with the last $g-(k+2)$ rows removed, we reduce to the case that \eqref{matrix: F_p^n minors} is a square matrix. For every $j \neq 1$, we subtract $\frac{\varepsilon_j}{\varepsilon_1}$ times the first row from the $j$-th row, which does not change the determinant of \eqref{matrix: F_p^n minors}. Setting $y_1 = x_1$ and $y_j = x_j - \frac{\varepsilon_j}{\varepsilon_1}x_1$, we obtain the matrix

    \begin{align*}
    \begin{pmatrix}
        y_1 & y_1^{\sigma^n} & \cdots & y_1^{\sigma^{kn}} &\varepsilon_1 \\
        y_{2} & y_2^{\sigma^{n}} & \cdots & y_2^{\sigma^{kn}} &0 \\
        \cdots & \cdots & \cdots & \cdots & \cdots \\
        y_g & y_g^{\sigma^{n}} & \cdots & y_g^{\sigma^{kn}} & 0
    \end{pmatrix}.
    \end{align*}

    Since $y_2,...,y_g$ are still linearly independent over $\F_{p^n}$, this matrix has non-zero determinant by Lemma \ref{Lemma: a=1 dieudonne lattices generators}.
\end{proof}

As an immediate consequence of Lemma \ref{Lemma value PiN subset M}, we recover the known fact that any Dieudonné lattice $M \in \mathcal C_{\Lambda_0}(\overline{\F}_p)$ is generated by some $v \in \Lambda_0$ that is only supported on finitely many basis vectors $e_{i,l}$, namely for $i \in \llbracket 1,g\rrbracket$ and $l \in \llbracket0,N-1\rrbracket$. 
\begin{definition}\label{def: coordinate space}
    Let $\mathbb A^{gN}$ be affine space with coordinates $\underline{a}=(a_{i,l})_{(i,l) \in \llbracket 1,g\rrbracket \times \llbracket 0,N-1\rrbracket}$ over $\overline{\F}_p$.
    We define the coordinate space of $\mathcal{C}_{\Lambda_0}$ to be the irreducible open subscheme $U \subset \mathbb A^{gN}$ defined by the open condition of part (ii) of Lemma \ref{Lemma: a=1 dieudonne lattices generators}.
\end{definition}
Consider the map of sets $v(-):\mathbb A^{gN}(\overline{\F}_p) \to \Lambda_0$ which is defined as
\begin{equation*}
    v(\underline{a})=\sum\limits_{i=1}^{g}\sum\limits_{l=0}^{N-1} [a_{i,l}]e_{i,l} \in \Lambda_0
\end{equation*}
for any $\underline{a}= (a_{i,l})_{(i,l) \in \llbracket 1,g\rrbracket \times \llbracket 0,N-1\rrbracket} \in \mathbb A^{gN}(\overline{\F}_p)$.

By Lemmas \ref{Lemma: a=1 dieudonne lattices generators} and \ref{Lemma value PiN subset M}, the induced map $U(\overline{\F}_p) \to \mathcal C_{\Lambda_0}(\overline{\F}_p)$ that sends $\underline{a} \in U(\overline{\F}_p)$ to $\mathcal D \cdot v(\underline{a})$, is surjective. In general, it is not injective because Dieudonné generators need not be unique. 

\begin{remark}\label{Remark: coordinate space upgraded}
    By closer analysis, one can drastically shrink the dimension of $\mathbb A^{gN}$ and upgrade the analogously defined map $U(\overline{\F}_p) \to \mathcal C_{\Lambda_0}(\overline{\F}_p)$ into a morphism of schemes that is bijective on perfect points, see \cite[4.4]{viehmann2008modulispaces} for details. For our purposes, it suffices that the map is surjective on $\overline{\F}_p$-points.
\end{remark}
 
\section{Generic automorphisms of the universal \texorpdfstring{$p$}{p}-divisible group}\label{Section: automorphisms of lattices}
As the title suggests, the goal of this section is to prove Theorem \ref{thmmain3}, i.e.~ to compute the generic automorphism group of the universal $p$-divisible group over $\M$.

We quickly recall the reduction steps from the last section. By Dieudonné theory, we may equivalently compute the generic automorphism group of Dieudonné lattices in the fixed isocrystal $\breve{N}$. We may reduce to considering the $a=1$ locus $\M^\circ$, and furthermore, by transitivity of the $J_{b_0}(\Q_p)$-action on irreducible components, to the irreducible component $\mathcal C_{\Lambda_0}$ labelled by the vertex lattice $\Lambda_0 \subset \breve{N}$.

\subsection{Congruence subgroups}

Recall our fixed identification $\Aut(\Lambda_0,F)=\GL_g(\O_D)$ based on the choice of basis $(e_{i,l})_{i,l}$ of $\breve{N}$. Recall that for any Dieudonné lattice $M \in \mathcal C_{\Lambda_0}(\overline{\F}_p)$ we have a canonical injection
\begin{equation*}
    \Aut(M,F) \hookrightarrow \Aut(\Lambda_0,F)=\GL_g(\O_D).
\end{equation*}

\begin{definition}
    For $s \geq 1$, we define the congruence subgroup of level $s$
    \begin{equation*}
        V_s:=1+\pi^s\Mat_g(\O_D)\subset \GL_g(\O_D).
    \end{equation*}
    Note that implicitly, $V_s$ depends on $g$, $m$ and $n$.
\end{definition}

By definition, elements of $\Aut(\Lambda_0,F)$ preserve the filtration $\Lambda_\bullet$. Therefore the projection $[\cdot]_s:\Lambda_0\rightarrow \Lambda_0/\Lambda_s$ induces the projection $$[\cdot]_s:\Aut(\Lambda_0,F)=\GL_g(\O_D)\rightarrow \GL_g(\O_D/\pi^{s})=\GL_g(\O_D)/V_{s}=\Aut(\Lambda_0/\Lambda_{s},F).$$

\begin{definition}Let $N=gm(n-m)-n+1$. If $n \nmid g$ we define
\begin{equation*}
    \Gamma_{\gen} := \Z_p^\times + \pi^N \Mat_g(\O_D)=\Z_p^\times.V_N,
\end{equation*}
and if $n \mid g$ we define
\begin{equation*}
    \Gamma_{\gen}:=\Z_p^\times + \pi^{N-1}\O_D+\pi^N\Mat_g(\O_D)=(\Z_p^\times + \pi^{N-1}\O_D).V_N.
\end{equation*}
\end{definition}

\begin{remark}\label{remAutUniform}
One can uniformly define $\Gamma_{\gen}$ as
\[\Gamma_{\gen} = \Z_p^\times + p^{\lceil\frac{N-1}{n}\rceil}\O_D + \pi^N\Mat_g(\O_D).\]
Indeed, note that $p^{\lceil \frac{N-1}{n}\rceil}=\pi^{n\lceil\frac{N-1}{n}\rceil}$, and $n\lceil\frac{N-1}{n}\rceil\geq N-1$ with equality if and only if $n\mid N-1$. Since $N-1=gm(n-m)-n$ and since $m$, $n$ are coprime, $n$ divides $N-1$ if and only if $n$ divides $g$. Therefore, if $n \nmid g$, then $p^{\lceil\frac{N-1}{n}\rceil}\O_D\subset \pi^N\Mat_g(\O_D)$.
\end{remark}

The main result of this section is the following, from which Theorem~\ref{thmmain3} then immediately follows.
\begin{theorem}\label{Theorem: generic automorphism group}
    There is a dense open subscheme $Y_{\Lambda_0} \subset \mathcal C_{\Lambda_0}$ such that for all $M \in Y_{\Lambda_0}(\overline{\F}_p)$
    \begin{equation*}
        \Aut(M,F) = \Gamma_{\gen}.
    \end{equation*}
\end{theorem}
Let us handle the obvious inclusion first, which holds without any genericity assumptions.
\begin{lemma}\label{lemma: Gamma_gen minimal}
    $\Gamma_{\gen}\subseteq \Aut(M,F)$ for any $M\in \mathcal C_{\Lambda_0}(\overline{\F}_p)$.
\end{lemma}

\begin{proof}
    Let $M\in\mathcal C_{\Lambda_0}(\overline{\F}_p)$. By our identification 
    $\Aut(\Lambda_0,F) =\GL_g(\mathcal O_D)$, an element $\gamma\in \GL_g(\mathcal O_D)$ lies in $\Aut(M,F)$ if and only if $\gamma(M)\subset M$.

    The map $[\cdot]_N:\Aut(\Lambda_0,F)\rightarrow \Aut(\Lambda_0/\Lambda_N)$ induces an exact sequence 
 \[	0 \rightarrow {V_N} \rightarrow {\Gamma_{\gen}} \rightarrow {[\Z_p^\times + p^{\lceil\frac{N-1}{n}\rceil}\O_D]_N} \rightarrow 0  \]

    By Lemma~\ref{Lemma value PiN subset M} we know that $V_N\subset \Aut(M,F)$. It is also clear that $\Z_p^\times \subset \Aut(M,F)$ is central (and by Remark~\ref{remark: scalars in Aut N}, these are all the scalar multiplications in $\Aut(M,F)$). Thus it remains to prove that $[1+p^{\lceil \frac{N-1}{n}\rceil}\mathcal O_D]_N\subset \Aut(M/\Lambda_N)$. 

    Since $M = \mathcal D \cdot v$ for some generator $v \in \Lambda_0$, it suffices to see that $[p^{\lceil \frac{N-1}{n}\rceil}\xi v]_N\in[\mathcal D\cdot v]_N$ for all $\xi\in \mathcal O_D,v\in\Lambda_0$. Indeed, let $\tilde{\xi} \in \breve{\Z}_p$ be a lift of $[\xi]_1 \in \O_D/\pi = \F_{p^n}$. Then
    \[[p^{\lceil \frac{N-1}{n}\rceil}\xi v]_N=[p^{\lceil \frac{N-1}{n}\rceil}\tilde \xi v]_N\in [\mathcal D\cdot v]_N\]
    because $p^{\lceil\frac{N-1}{n}\rceil}\xi \equiv p^{\lceil\frac{N-1}{n}\rceil}\tilde{\xi} \mod \pi^N$.
\end{proof}

The following lemma shows that in order to prove Theorem \ref{Theorem: generic automorphism group}, we essentially have to find for any $\gamma \in \GL_g(\O_D) \setminus \Gamma_{\gen}$ a Dieudonné lattice that is not stable under $\gamma$.
\begin{lemma}\label{lemma: Y N is open}
    The set of $\overline{\F}_p$-points
    \begin{equation*}
        Y_{\Lambda_0}(\overline{\F}_p)=\{M\in \mathcal{C}_{\Lambda_0}(\overline{\F}_p) \mid \Aut(M,F)=\Gamma_{\gen}\}
    \end{equation*}
    defines an open subscheme $Y_{\Lambda_0}$ of $\mathcal C_{\Lambda_0}$. 
    
    Moreover it is dense in $\mathcal C_{\Lambda_0}$ if and only if for all $\gamma\in \Aut(\Lambda_0,F)\setminus \Gamma_{\gen}$ there exists a Dieudonné lattice $M\in \mathcal C_{\Lambda_0}(\overline{\F}_p)$ such that $\gamma\notin \Aut(M,F)$.
\end{lemma}
\begin{proof}
   Since the action of $J_{b_0}$ on the Rapoport-Zink space $\M$ is continuous, any $\gamma\in \Aut(\Lambda_0,F)$ induces a continuous map on $\mathcal C_{\Lambda_0}$. Note that $M \in \mathcal C_{\Lambda_0}$ is fixed by $\gamma$ if and only if $\gamma\in \Aut(M,F)$. Since $\mathcal C_{\Lambda_0}$ is separated, its set of $\gamma$-fixed points is Zariski closed so
    $$Y_\gamma(\overline{\F}_p)=\{M\in \mathcal C_{\Lambda_0}(\overline{\F}_p) \mid \gamma \notin \Aut(M,F)\}$$ are the $\overline{\F}_p$-points of an open subscheme $Y_\gamma$ of $\mathcal C_{\Lambda_0}$.

    Since $1+\pi^N\Mat_g(\mathcal O_D)$ is finite index in $\GL_g(\O_D)$, so is $\Gamma_{\gen}$. Let $(\gamma_i)_{i\in I}$ be a choice of coset representatives of $(\GL_g(\O_D)/\Gamma_{\gen})\setminus \{\Gamma_{\gen}\}$. Since $I$ is finite, $\cap_{i\in I} Y_{\gamma_i}$ is an open subscheme of $\mathcal C_{\Lambda_0}$ and by Lemma~\ref{lemma: Gamma_gen minimal} 
    $$\big(\bigcap\limits_{i\in I} Y_{\gamma_i} \big)(\overline{\F}_p)=\{M\in \mathcal C_{\Lambda_0}(\overline{\F}_p)\mid \Aut(M,F)=\Gamma_{\gen}\}=Y_{\Lambda_0}(\overline{\F}_p).$$
    
    This proves the first claim. To prove the second claim, note that $Y_{\Lambda_0} = \cap_{i \in I} Y_{\gamma_i}$ is dense in $\mathcal C_{\Lambda_0}$ if and only if $Y_{\gamma_i}$ is for each $i \in I$. Since $\mathcal C_{\Lambda_0}$ is irreducible, each $Y_{\gamma_i}$ is dense if and only if it is non-empty, i.e.~if and only if each $Y_{\gamma_i}$ admits an $\overline{\F}_p$-point. The second claim follows immediately.
 \end{proof}

Therefore, to prove Theorem~\ref{Theorem: generic automorphism group} it is enough to construct for each $\gamma\in \GL_g(\O_D)\setminus \Gamma_{\gen}$ a Dieudonné lattice $M\in\mathcal C_{\Lambda_0}(\overline{\F}_p)$ such that $\gamma \notin \Aut(M,F)$. The main technical result that we are going to use is the next proposition.

\begin{proposition} \label{proposition: U delta density max case}
Let $s_0\in \llbracket 0,N-1\rrbracket$ and $\delta\in \Hom_{\breve{\Z}_p}(\Lambda_0,\Lambda_{s_0})$, and suppose that either:
\begin{enumerate}
    \item $\delta\tau_1=\tau_1\delta$ and $[\delta]_{s_0+1}\notin [p^{\lceil\frac{s_0}{n}\rceil}\breve{\Z}_p]_{s_0+1}$ or
    \item $n\nmid s_0$ and $[\delta]_{s_0+1}\neq0$.
\end{enumerate}
Then the set
\begin{equation}\label{eq U delta definition}
	U_\delta:=\{\underline{a}\in U(\overline{\F}_p) \mid \delta(v(\underline{a}))\notin \D\cdot v(\underline{a})\}
\end{equation} is non-empty.
\end{proposition}
\begin{remark}
    Note that a priori we do not assume $\delta$ to be compatible with $F$, we only make the mentioned weaker assumptions.
\end{remark}
The proof of Proposition~\ref{proposition: U delta density max case} is supplied in Section~\ref{subsection: minors computation}. For now, let us deduce Theorem~\ref{Theorem: generic automorphism group} from it.

\begin{proof}[Proof that Proposition~\ref{proposition: U delta density max case} implies Theorem~\ref{Theorem: generic automorphism group}]
    By Lemma~\ref{lemma: Y N is open} it suffices to produce, for each $\gamma \in \GL_g(\O_D)\setminus \Gamma_{\gen}$, a Dieudonné lattice $M\in \mathcal C_{\Lambda_0}(\overline{\F}_p)$ such that $\gamma\notin \Aut(M,F)$. Since $\mathcal C_{\Lambda_0}(\overline{\F}_p)=\mathcal D\cdot v(U(\overline{\F}_p))$, it is equivalent to finding $\underline a\in U(\overline{\F}_p)$ such that $\gamma(v(\underline a))\notin \mathcal D\cdot v(\underline a)$.

    Let us now fix $\gamma\in \GL_g(\O_D)\setminus \Gamma_{\gen}$. Then $[\gamma]_N\notin [\Z_p^\times + p^{\lceil\frac{N-1}{n}\rceil}\O_D]_N$. Let $s_0\in \llbracket 0, N-1\rrbracket$ be minimal such that $[\gamma]_{s_0+1}\notin [\Z_p^\times + p^{\lceil\frac{N-1}{n}\rceil}\O_D]_{s_0+1}$. Because $n\lceil\frac{N-1}{n}\rceil\geq N-1\geq s_0$, we have $[p^{\lceil\frac{N-1}{n}\rceil}\O_D]_{s_0}=0$. Therefore minimality of $s_0$ implies that there exists $\lambda\in \Z_p^\times$ such that $[\gamma]_{s_0}=[\lambda]_{s_0}$. Consider the morphism $\delta_1:=\gamma-\lambda \in \Hom_{\breve{\Z}_p}(\Lambda_0,\Lambda_{s_0})$ (using that $[\delta_1]_{s_0}=0$). Since $\gamma$ and $\lambda$ both commute with $\tau_1$, so does $\delta_1$.
    
    We make the following case distinction.

    \textbf{Case 1:} Suppose that $[\gamma]_{s_0+1}\notin [\Z_p^\times + p^{\lceil\frac{s_0}{n}\rceil}\O_D]_{s_0+1}$.  Then $[\delta_1]_{s_0+1}\notin [p^{\lceil\frac{s_0}{n}\rceil}\breve{\Z}_p]_{s_0+1}$. Thus by part (i) of Proposition~\ref{proposition: U delta density max case}, $U_{\delta_1}=\{\underline{a} \in U(\overline{\F}_p) \mid \delta_1(v(\underline{a})) \notin \mathcal D \cdot v(\underline{a})\}$ is non-empty. Moreover, since Dieudonné lattices are additive groups stable by scalar multiplication, it is clear that it satisfies 
    \begin{equation*}
        U_{\delta_1} = \{\underline{a} \in U(\overline{\F}_p) \mid \gamma(v(\underline{a})) \notin \mathcal D \cdot v(\underline{a})\}.
    \end{equation*}
    Hence for such $\gamma$ there exists an $M\in \mathcal C_{\Lambda_0}(\overline{\F}_p)$ satisfying $\gamma \notin \Aut(M,F)$.

    \textbf{Case 2:} Suppose now that $[\gamma]_{s_0+1}\in [\Z_p^\times + p^{\lceil\frac{s_0}{n}\rceil}\O_D]_{s_0+1}$ (in particular $s_0+1 < N$). By assumption on $s_0$ we know that $[\gamma]_{s_0+1}\notin [\Z_p^\times]_{s_0+1}$, which in turn forces $[p^{\lceil\frac{s_0}{n}\rceil}\O_D]_{s_0+1}\neq 0$ and thus $n \mid s_0$. Since $p^\frac{s_0}{n}\pi\in \operatorname{ker}([\cdot]_{s_0+1})$, we have $[p^{\frac{s_0}{n}}\O_D]_{s_0+1}=[p^{\frac{s_0}{n}}\Z_{p^n}]_{s_0+1}$. Thus we can choose $\mu\in \Z_{p^n}$ such that $[\gamma]_{s_0+1}=[\lambda+p^{\frac{s_0}{n}}\mu]_{s_0+1}$. Recall that $\tau_1$ is $\sigma^b$-linear where $b$ is an inverse of $m$ mod $n$ (in particular it is coprime to $n$).
    By our assumption on $s_0$ we know that $[p^{\frac{s_0}{n}}\mu]_{s_0+1}\notin [\Z_p]_{s_0+1}$, so that $[\mu]_1\in \F_{p^n}\setminus \F_p$. Therefore $[\mu-\mu^{\sigma^b}]_1=[\mu]_1-([\mu]_1)^{\sigma^b}\neq 0$, that is to say $\operatorname{val}(\mu-\mu^{\sigma^b})=0$.

    Let $\delta_2=\delta_1-p^{\frac{s_0}{n}}\mu\in \Hom_{\breve{\Z}_p}(\Lambda_0,\Lambda_{s_0+1})$. By the same reasoning as before we have
    \begin{equation*}
        U_{\delta_2}= \{\underline{a} \in U(\overline{\F}_p) \mid \gamma(v(\underline{a})) \notin \mathcal D \cdot v(\underline{a})\}.
    \end{equation*}
    We want to apply part (ii) of Proposition~\ref{proposition: U delta density max case} to argue as in the previous case that $U_{\delta_2}$ is non-empty, which would conclude the proof. The fact that $n \mid s_0$ implies that $n \nmid s_0+1$, so in particular it suffices to check that $[\delta_2]_{s_0+2}\neq 0$ in order to apply the Proposition to $\delta_2$ at level $s_0+1$ (recall that $s_0+1 < N$).

    We now prove that $[\delta_2]_{s_0+2}\neq 0$ by evaluating it on the basis vectors $e_{i,1}$. Since $\tau_1$ is $\sigma^b$-linear and commutes with $\delta_1$, we have
    \[\delta_2(e_{i,1})=(\delta_1-p^\frac{s_0}{n}\mu)(e_{i,1})=\tau_1(\delta_1-p^\frac{s_0}{n}\mu^{\sigma^{-b}})(e_{i,0}).\]

    Therefore, $[\delta_2(e_{i,1})]_{s_0+2}\neq 0$ if and only if $[(\delta_1-p^\frac{s_0}{n}\mu^{\sigma^{-b}})(e_{i,0})]_{s_0+1}\neq 0$. Since $[\delta_1]_{s_0+1}=[p^\frac{s_0}{n}\mu]_{s_0+1}$, we have $[(\delta_1-p^\frac{s_0}{n}\mu^{\sigma^{-b}})(e_{i,0})]_{s_0+1}=[p^\frac{s_0}{n}(\mu-\mu^{\sigma^{-b}})e_{i,0}]_{s_0+1}$. By the above discussion, $\operatorname{val}(p^\frac{s_0}{n}(\mu-\mu^{\sigma^{-b}}))=s_0$ and thus $[p^\frac{s_0}{n}(\mu-\mu^{\sigma^{-b}})e_{i,0}]_{s_0+1}\neq 0$. This proves the claim and concludes the proof.
\end{proof}

\subsection{Construction of open dense subsets of parameters}\label{subsection: minors computation}
The aim of this subsection is to supply the proof of Proposition~\ref{proposition: U delta density max case}. 

In the proof, we leverage the amount of freedom we have to maneuver in $\Lambda_{s_0}/\Lambda_{s_0+1}$. Namely, fix $\delta$ as in the proposition. Then by Lemma~\ref{Lemma: basis of level s} the inclusion
\begin{equation*}
    M_{s_0}/M_{s_0+1} \subseteq \Lambda_{s_0}/\Lambda_{s_0+1}
\end{equation*}
is either strict for all $M\in \mathcal C_{\Lambda_0}$ or an equality for all $M\in \mathcal C_{\Lambda_0}$. When the inclusion is strict, we are going to have enough room to find elements $\underline{a}\in U(\overline{\F}_p)$ such that, for $M=\mathcal D \cdot v(\underline{a})$,  $[\delta(v(\underline{a}))]_{s_0+1} \notin M_{s_0}/M_{s_0+1}$, which implies $\delta(v(\underline{a}))\notin M$. This is the case whenever $s_0 < (g-1)m(n-m)$, and also sporadically for $(g-1)m(n-m)+1 \leq s_0 \leq N-1$ by Lemma \ref{Lemma value PiN subset M}. This sporadicity will cause us some additional work.

\begin{example}\label{example: explicit parameters}
To illustrate this in a simple example, suppose that $n \geq 3$, that $s_0=m<n-m$ and that $\delta\in \Hom_{\breve{\Z}_p}(\Lambda_0,\Lambda_m)$ is defined by $\delta(e_{i,l})=e_{i,l+m}$. Let $\underline{a}=(a_{i,l})_{i,l}\in U(\overline{\F}_p)$ and $M= \mathcal D \cdot v(\underline{a})$. By Lemma~\ref{Lemma: basis of level s}, the $\overline{\F}_p$-vector space $M_m/M_{m+1}$ is one-dimensional and generated by $[F(v(\underline{a}))]_{m+1}=\sum\limits_{i=1}^g a_{i,0}^\sigma [e_{i,m}]_{m+1}$.

\begin{enumerate}
    \item If $g=2$, then $M_m/M_{m+1} \subsetneq \Lambda_m/\Lambda_{m+1}$ has codimension $1$, so we may try to find a parameter $\underline{a}$ for which 
        \[ [\delta(v(\underline{a}))]_{m+1}=a_{1,0}[e_{1,m}]_{m+1}+a_{2,0}[e_{2,m}]_{m+1}\]
    does not lie in $M_m/M_{m+1}$. It is collinear to 
        \[[F(v(\underline{a}))]_{m+1}=a_{1,0}^\sigma[e_{1,m}]_{m+1}+a_{2,0}^\sigma[e_{2,m}]_{m+1}\]
    if and only if $a_{1,0}a_{2,0}^\sigma -a_{2,0}a_{1,0}^\sigma = 0$. The latter does not hold if $a_{1,0},a_{2,0}$ are linearly independent over $\F_p$, which is the case when $\underline{a}\in U(\overline{\F}_p)$.

    \item Suppose now that $g=1$. Then $M_m/M_{m+1} = \Lambda_m/\Lambda_{m+1}$ and we cannot apply the preceding argument. For sake of exposition, suppose that $n-m>m+1$ which implies $\mathcal B_{m+1}=\emptyset$ and thus by Lemma~\ref{Lemma: basis of level s} that $M_{m+1}/M_{m+2} =0$. So in order to find a parameter $\underline{a}$ such that $\delta(v(\underline{a})) \notin M$, it suffices to construct a linear combination $\tilde{v} \in \Lambda_{m+1}$ of $\delta(v(\underline{a}))$ with an element of $M$ such that $[\tilde{v}]_{m+2} \neq 0$. Since
    \[
        [\delta(v(\underline{a}))]_{m+1}=a_{1,0}[e_m]_{m+1}=\frac{a_{1,0}}{a_{1,0}^\sigma}[F(v(\underline{a}))]_{m+1},
    \]
    we have $\delta(v(\underline{a}))-[\frac{a_{1,0}}{a_{1,0}^\sigma}]F(v(\underline{a}))\in \Lambda_{m+1}$ and 
    \[
        [\delta(v(\underline{a}))-[\frac{a_{1,0}}{a_{1,0}^\sigma}]F(v(\underline{a}))]_{m+2}=(a_{1,1}-\frac{a_{1,0}}{a_{1,0}^\sigma}a_{1,1}^\sigma)[e_{m+1}]_{m+2}.
    \]
    This is non-zero whenever $a_{1,0}$ and $a_{1,1}$ are linearly independent over $\F_p$, so it defines a non-empty (and even Zariski dense) subset of $U(\overline{\F}_p)$. 
\end{enumerate}

\end{example}
The first example illustrates our strategy of proof whenever $|\mathcal B_{s_0}|<g$. The second example illustrates how to deal with the remaining cases where $|\mathcal B_{s_0}|=g$. We then have to truncate at the smallest level $s_1>s_0$ such that $|\mathcal B_{s_1}|<g$. We know that such an $s_1$ exists by Lemma~\ref{Lemma value PiN subset M} because $s_0<N$.

In these two examples, finding satisfactory parameters boils down to taking the non-vanishing locus of certain functions on $\A^{gN}(\overline{\F}_p)$.  In general for $|\mathcal B_{s_0}|<g$ these functions can be thought of as polynomials whose exponents are powers of $p$ (potentially negative because $V$ is $\sigma^{-1}$-linear). They are defined as minors of matrices, compare Lemma~\ref{Lemma: minors of delta matrix}. We can then prove the result in this case, which we do in the intermediary Proposition~\ref{proposition: U delta density}.

In the case  $|\mathcal B_{s_0}|=g$ these functions are in general more involved, but they are still defined as minors of matrices. We prove that they are non-zero in Lemma~\ref{Lemma: Minors of delta matrix bis}, and then conclude with the proof of Proposition~\ref{proposition: U delta density max case}.
\begin{notation} To prepare these generalizations, let us introduce the following notations.

    \begin{enumerate}

        \item The symbol $g'$ denotes an element of $\llbracket 0,g-1\rrbracket$, and $r_1,\dots,r_{g'}$ denote pairwise distinct integers.
        \item The symbol $J$ denotes a subset of $\llbracket 1,g\rrbracket$ of cardinal $|J|=g'+1$.
        \item The symbol $\mathfrak S_J$ denotes the set of bijections $\llbracket 0,g'\rrbracket\rightarrow J$. We let $\epsilon:\mathfrak S_J\rightarrow\{\pm 1\}$ denote a signature function on $\mathfrak S_J$.
        \item For any $(i,l,f)\in \llbracket 1,g\rrbracket\times \llbracket 0,N-1\rrbracket\times \mathfrak S_J$, the symbol $P_{i,l,f}$ denotes the monomial
	\begin{equation*}
		P_{i,l,f}=X_{i,l}\prod\limits_{k=1}^{g'} X^{p^{r_k}}_{f(k),0}.
	\end{equation*}
    \end{enumerate}
\end{notation}

We let $\mathcal P$ denote the ring $\mathcal P=\overline{\F}_p\big[X^{p^{-n}}_{i,l}: i\in \llbracket1,g\rrbracket, \, l\in \llbracket0,N-1\rrbracket,\, n\geq 0\big]$. Elements of $\mathcal P$ can be evaluated on points of $\A^{gN}$, setting that the evaluation of $X_{i_0,l_0}^{p^n}$ on $\underline{a}=(a_{i,l})_{i,l}$ is given by $a_{i_0,l_0}^{\sigma^n}$. The non-vanishing locus of a non-zero element of $\mathcal P$ then defines an open dense subset of $\A^{gN}$, and in particular a non-empty subset of $U(\overline{\F}_p)$.

We identify the monomials $P_{i,l,f}$ with elements of $\mathcal P$.

Each choice of variable $X_{i,l}$ induces a degree function $\operatorname{deg}_{X_{i,l}}: \mathcal P\rightarrow \Z_{\geq 0}[p^{-1}]$. For instance, $\operatorname{deg}_{X_{i,l}}(P_{i,l,f})$ is either equal to $1$ (if $l\neq 0$ or $i\notin J$) or to $1+p^{r_{f^{-1}(i)}}$ (if $l=0$ and $i\in J$). For $P\in \mathcal P$ we have $\operatorname{deg}_{X_{i,l}}(P)=0$ if and only if $X_{i,l}$ does not contribute to $P$.

\begin{lemma}\label{Lemma: equivalence of monomials}
For two distinct tuples $(i,l,f),(i',l',f')\in \llbracket 1,g\rrbracket\times \llbracket0,N-1\rrbracket\times \mathfrak S_J$, we have $$P_{i,l,f}=P_{i',l',f'}$$ if and only if the following conditions are all satisfied:
\begin{enumerate}
	\item There exists a (necessarily unique) $k_0\in \llbracket 1,g'\rrbracket$ such that $r_{k_0}=0$,
	\item $l=l'=0$,
	\item $f'=f\circ \tau_{(0,k_0)}$ where $\tau_{(0,k_0)}$ denotes the transposition exchanging $0$ and $k_0$,
	\item $i=f(0)$ and $i'=f(k_0)$.
\end{enumerate}
In particular, if $0\notin \{r_1,\dots,r_{g'}\}$, the monomials $P_{i,l,f}$ are pairwise distinct.
\end{lemma}

\begin{proof}
     Suppose first that $0\notin \{r_1,\dots,r_{g'}\}$ then the numbers $p^{r_1},\dots,p^{r_{g'}}, 1+p^{r_1},\dots, 1+p^{r_{g'}}$ are all distinct. Moreover if $f'\neq f$ there is $k\in \llbracket 1,g'\rrbracket$ with $f(k)\neq f'(k)$ (as bijections with the same finite image, they cannot differ only at $0$). We deduce that $k':=f'^{-1}(f(k))\neq k$ and thus $$\operatorname{deg}_{X_{f(k),0}}(P_{i',l',f'})=\delta_{i'=f(k),l'=0}+p^{r_{k'}}\notin \{p^{r_k},1+p^{r_k}\}\ni \operatorname{deg}_{X_{f(k),0}}(P_{i,l,f}),$$ so $P_{i,l,f}\neq P_{i',l',f'}$. Similarly if $f=f'$ but $(i,l)\neq (i',l')$, then $$\operatorname{deg}_{X_{i,l}}(P_{i,l,f})\in \{1,1+p^{r_1},\dots,1+p^{r_{g'}}\}\not\ni\operatorname{deg}_{X_{i,l}}(P_{i',l',f'}).$$ So if Condition (1) is not satisfied, the monomials $P_{i,l,f}$ are pairwise distinct.

    Suppose now that $0\in \{r_1,\dots,r_{g'}\}$. Upon relabeling, suppose without loss of generality that $r_1=0$. Since the $r_i$ are pairwise distinct, $r_k\neq 0$ for all $k\in \llbracket 2,g'\rrbracket$. The same arguments as previously show that $P_{i,l,f}\neq P_{i',l',f'}$ if either $l$ or $l'$ is nonzero, or if $f'|_{\llbracket 2,g'\rrbracket}\neq f|_{\llbracket 2,g'\rrbracket}$.
    Note that, since $f$ and $f'$ have the same image, $f|_{\llbracket 2,g'\rrbracket}=f'|_{\llbracket 2,g'\rrbracket}$ if and only if $f'\in \{f,f\circ \tau_{(0, 1)}\}$. If $f=f'$ but $i\neq i'$ then $P_{i',0,f'}=\frac{X_{i',0}}{X_{i,0}}P_{i,0,f}\neq P_{i,0,f}$. 
    
    We have thus proven that $P_{i,l,f}\neq P_{i',l',f'}$ if conditions (1), (2) or (3) are not satisfied. Suppose now that conditions (1)--(3) are all satisfied. We then have \begin{align*}P_{i,0,f}&=X_{i,0}X^{r_1}_{f(1),0}\prod\limits_{k=2}^{g'} X_{f(k),0}^{r_k}, \\ P_{i',0,f'}&=X_{i',0}X^{r_1}_{f'(1),0}\prod\limits_{k=2}^{g'} X_{f'(k),0}^{r_k}=X_{i',0}X^{r_1}_{f(0),0}\prod\limits_{k=2}^{g'} X_{f(k),0}^{r_k}.\end{align*} Since $r_1=0$, these two monomials agree if and only if $i=f(0)$ and $i'=f(1)$, i.e. if and only if Condition (4) is satisfied.
\end{proof}

\begin{lemma}\label{Lemma: minors of delta matrix}Let $(\delta_{i,l}^j)_{(i,j,l)\in \llbracket1,g\rrbracket^2\times \llbracket0,N-1\rrbracket}$ be a family of elements of $\overline{\F}_p$. Define the $(g'+1)\times g$-matrix $A_1(\underline{X})\in \Mat_{(g'+1)\times g}(\mathcal P)$ as:
\begin{equation*}A_1(\underline{X}):=
    \begin{pmatrix}
    
		\sum\limits_{l=0}^{N-1}\sum\limits_{i=1}^g\delta_{i,l}^1 X_{i,l} & \dots & \sum\limits_{l=0}^{N-1}\sum\limits_{i=1}^g\delta_{i,l}^g X_{i,l}  \\
		X_{1,0}^{p^{r_1}} & \dots & X_{g,0}^{p^{r_1}}\\
		\dots & \dots & \dots
		\\
		X_{1,0}^{p^{r_{g'}}}& \dots &X_{g,0}^{p^{r_{g'}}}
	\end{pmatrix}
    \end{equation*}
    Then the matrix $A_1$ admits a non-zero $(g'+1)$-minor if and only if one of the following conditions is satisfied:
    \begin{enumerate}
        \item There exists indices $(i,j,l)$ with $l>0$ or $i\neq j$, such that $\delta^j_{i,l}\neq 0$,
        \item there exists indices $(i,j)$ such that $\delta^i_{i,0}\neq \delta^j_{j,0}$,
        \item there exists an index $i$ such that $\delta^i_{i,0}\neq 0$ and $0\notin \{r_1,\dots,r_{g'}\}$.
    \end{enumerate}

\end{lemma}
\begin{proof}

For $J\subset \llbracket 1,g\rrbracket$ a set of indices of cardinal $g'+1$, the minor obtained by only taking the columns of indices in $J$ writes out \begin{equation}\label{eq formula P_J 1}P_J=\sum\limits_{f\in\mathfrak S_J}\sum\limits_{l=0}^{N-1}\sum\limits_{i=1}^g\epsilon(f)\delta_{i,l}^{f(0)}P_{i,l,f}.\end{equation}

Fix a tuple $(i,l,f)\in \llbracket 1,g\rrbracket\times \llbracket 0,N-1\rrbracket\times \mathfrak S_J$.

By Lemma~\ref{Lemma: equivalence of monomials}, if either $l\neq 0$, $i\neq f(0)$ or $0\notin\{r_1,\dots,r_{g'}\}$, then $P_{i,l,f}$ is distinct from all the other monomials $P_{i',l',f'}$, so $P_J=0$ implies $\delta^{f(0)}_{i,l}=0$ in this case. 

On the other hand if $0\in \{r_1,\dots,r_{g'}\}$ (say $0=r_{k_0}$), $l=0$ and $i=f(0)$, then by Lemma~\ref{Lemma: equivalence of monomials} the only other tuple $(i',l',f')$ with $P_{f(0),0,f}=P_{i',l',f'}$ is $(f(k_0),0,f\circ \tau)$ where $\tau$ is the transposition exchanging $0$ and $k_0$. Then $\epsilon(f\circ \tau)=-\epsilon(f)$ so the coefficient of the monomial $P_{f(0),0,f}$ in Formula~\eqref{eq formula P_J 1} is $\epsilon(f) (\delta_{f(0),0}^{f(0)}-\delta_{f(k_0),0}^{f(k_0)})$. 

Running over all triples $(i,l,f)$, it follows that $P_J\neq 0$ if and only if one of the conditions (1), (2) or (3) are satisfied for indices $i,j\in J$ and $l\geq 0$. Running over all choices of $J$, the result follows.
\end{proof}

At this stage, we can already prove a weaker analog of Proposition~\ref{proposition: U delta density max case}:

\begin{proposition}\label{proposition: U delta density}
Let $s_0\in \Z_{\geq 0}$ be such that $|\mathcal B_{s_0}|<g$. Let $\delta \in \Hom_{\breve{Z}_p}(\Lambda_0,\Lambda_{s_0})$. Then, if $[\delta]_{s_0+1}\notin [p^{\lceil\frac{s_0}{n}\rceil}\breve{\Z}_p]_{s_0+1}$, the set
\begin{equation*}
	U_\delta=\{\underline{a}\in U(\overline{\F}_p) \mid \delta(v(\underline{a}))\notin \D\cdot v(\underline{a})\}
\end{equation*} is non-empty.
\end{proposition}
\begin{proof}
The quotient $\Lambda_{s_0}/\Lambda_{s_0+1}$ is an $\overline{\F}_p$-vector space freely generated by the family $([e_{i,s_0}]_{s_0+1})_{i\in \llbracket1,g\rrbracket}$.
Let $(\delta_{i,l}^{j})_{i,j,l}$ denote the $\overline{\F}_p$-matrix coefficients of $[\delta]_{s_0+1}\in \Hom_{\overline{\F}_p}(\Lambda_0/p\Lambda_0,\Lambda_{s_0}/\Lambda_{s_0+1})$, explicitly defined by $$[\delta(e_{i,l})]_{s_0+1}=\sum\limits_{j=1}^g\delta_{i,l}^{j}[e_{j,s_0}]_{s_0+1}.$$
	
Since $pe_{i,l}=e_{i,l+n}$, in terms of matrix coefficients, $[\delta]_{s_0+1}\notin[p^{\lceil\frac{s_0}{n}\rceil}\breve{\Z}_p]_{s_0+1}$ if and only if one of the following conditions is satisfied:
    \begin{enumerate}
        \item there is $l>0$ with $(\delta_{i,l}^j)_{i,j}$ non-zero,
        \item $n \nmid s_0$ and $(\delta_{i,0}^{j})_{i,j}$ is non-zero,
        \item $n\mid s_0$ and $(\delta_{i,0}^{j})_{i,j}$ is not diagonal.
    \end{enumerate}
    
	Set $g'=|\mathcal B_{s_0}|$ and choose a labeling $\{(a_k,b_k)\mid k\in \llbracket1,g'\rrbracket\}=\mathcal B_{s_0}$. For each $k\in \llbracket1,g'\rrbracket$, let $r_k=a_k-b_k$, so that the operator $F^{a_k}V^{b_k}$ is $\sigma^{r_k}$-linear. In particular, note that $0\in \{r_1,\dots,r_{g'}\}$ if and only if $n$ divides $s_0$.

    For $\underline{a}=(a_{i,l})_{(i,l)}\in U(\overline{\F}_p)$, by Lemma~\ref{Lemma: basis of level s} the vector space $[\D\cdot v(\underline{a})\cap \Lambda_{s_0}]_{s_0+1}$ is freely generated by $$([F^aV^b(v(\underline{a}))]_{s_0+1})_{(a,b)\in \mathcal B_{s_0}}=(\sum\limits_{i=1}^g a_{i,0}^{p^{r_k}}[e_{i,s_0}]_{s_0+1})_{k\in \llbracket1,g'\rrbracket}.$$
Then, expressing the condition $[\delta(v(\underline{a}))]_{s_0+1}\notin [\D\cdot v(\underline{a})]_{s_0+1}$ in the basis $([e_{i,s_0}]_{s_0+1})_{i\in \llbracket1,g\rrbracket}$ of $\Lambda_{s_0}/\Lambda_{s_0+1}$, we see that
    \begin{equation*}
        [\delta(v(\underline{a}))]_{s_0+1}\notin [\D\cdot v(\underline{a})]_{s_0+1}\iff \operatorname{rk}(A_1(\underline{a}))=g'+1 .
    \end{equation*}
    This is satisfied if and only if $P_J(\underline{a})\neq 0$ for some $(g'+1)$-minor $P_J$ of $A_1$. By Lemma~\ref{Lemma: minors of delta matrix}, one of these minors is non-zero if one of the conditions (1), (2), (3) above is satisfied. Thus $\operatorname{rk}(A_1(\underline{a}))=g'+1$ is satisfied on a Zariski open dense subset of $\A_{g,N}$ if $[\delta]_{s_0+1}\notin [p^{\lceil\frac{s_0}{n}\rceil}\breve{\Z}_p]_{s_0+1}$. Since $U$ is also open and dense in $\A_{g,N}$, the result follows.
    \end{proof}

As explained through Example~\ref{example: explicit parameters}.(2), in order to deal with levels $s_0<N$ such that $|\mathcal B_{s_0}|=g$, we need to truncate at a higher level $s_1$ which satisfies $|\mathcal B_{s_1}|<g$. Lemma~\ref{Lemma: Minors of delta matrix bis} below is an analog of Lemma~\ref{Lemma: minors of delta matrix} which allows us to do that.

\begin{definition}
    
We let $\mathcal F=\operatorname{Frac}(\mathcal P)$ denote the fraction field of $\mathcal P$. The degree functions $\operatorname{deg}_{X_{i,l}}$ naturally extend to $\mathcal F$, by $\operatorname{deg}_{X_{i,l}}(\frac{P}{Q})=\operatorname{deg}_{X_{i,l}}(P)-\operatorname{deg}_{X_{i,l}}(Q)$. 

For a given variable $X_{i,l}$, we say that an element $f\in \mathcal F$ is affine with respect to $X_{i,l}$ if it is of the form $f=\frac{P_1+X_{i,l}P_2}{Q}$ with $P_1,P_2,Q\in \mathcal P$ and $\operatorname{deg}_{X_{i,l}}(P_1)=\operatorname{deg}_{X_{i,l}}(P_2)=\operatorname{deg}_{X_{i,l}}(Q)=0$.
\end{definition}

\begin{lemma}\label{Lemma: Minors of delta matrix bis}
	Let $s_0< s_1\in \Z_{\geq 0}$. For any $s\in\llbracket s_0,s_1-1\rrbracket$ and any $\alpha\in \mathcal B_s$, let $r(\alpha)\in\Z$ and let $\tilde{P}_\alpha\in \mathcal F$. Suppose that the $\tilde{P}_\alpha$ are affine in all variables of level $l\geq s_1-s_0$.
	
    Let $(\delta_{i,l}^j)_{(i,j,l)\in \llbracket1,g\rrbracket^2\times \llbracket0,N-1\rrbracket}$ be a family of elements of $\overline{\F}_p$, and define a $(g'+1)\times g$-matrix $A_2(\underline{X})\in \Mat_{(g'+1)\times g}(\mathcal F)$ as:
	
	\begin{equation*}A_2(\underline{X}):=
		\begin{pmatrix}
			
			\sum\limits_{l=0}^{N-1}\sum\limits_{i=1}^g\delta_{i,l}^1 X_{i,l}-\sum\limits_{s=s_0}^{s_1-1}\sum\limits_{\alpha\in \mathcal B_s} \tilde{P}_\alpha X_{1,s_1-s}^{p^{r(\alpha)}} & \dots & \sum\limits_{l=0}^{N-1}\sum\limits_{i=1}^g\delta_{i,l}^g X_{i,l} -\sum\limits_{s=s_0}^{s_1-1}\sum\limits_{\alpha\in \mathcal B_s} \tilde{P}_\alpha X_{g,s_1-s}^{p^{r(\alpha)}} \\
			X_{1,0}^{p^{r_1}} & \dots & X_{g,0}^{p^{r_1}}\\
			\dots & \dots & \dots
			\\
			X_{1,0}^{p^{r_{g'}}}& \dots &X_{g,0}^{p^{r_{g'}}}
		\end{pmatrix}.
	\end{equation*}
Then, if one of the following conditions holds, the matrix $A_2(\underline{X})$ admits a non-zero $(g'+1)$-minor:
\begin{enumerate}
    \item There exists $r\in \Z\setminus\{0\}$ such that $\sum\limits_{\alpha\in \mathcal B_{s_0}, r(\alpha)=r}\tilde{P}_\alpha\neq 0$.
	\item The matrix $(\delta_{i,s_1-s_0}^j)_{(i,j)\in\llbracket1,g\rrbracket}$ is not diagonal.

\end{enumerate}
\end{lemma}
\begin{proof}
	For $J\subset \llbracket1,g\rrbracket$ a set of indices of cardinal $g'+1$, let $P_J$ denote the associated $g'+1$-minor of $A_2(\underline{X})$. Then $P_J$ writes out explicitly:
	\begin{equation*}
		P_J=\sum\limits_{f\in \mathfrak S_J}\epsilon(f)\bigg(\sum\limits_{l=0}^{N-1}\sum\limits_{i=1}^g\delta_{i,l}^{f(0)}P_{i,l,f}-\sum\limits_{s=s_0}^{s_1-1} \sum\limits_{\alpha\in \mathcal B_s}\tilde{P}_\alpha X_{f(0),s_1-s}^{p^{r(\alpha)}}\prod\limits_{k=1}^{g'}X^{p^{r_k}}_{f(k),0}\bigg).	
	\end{equation*}
We decompose $P_J$ as a (fractional) polynomial in the variables $(X_{i,s_1-s_0})_{i\in \llbracket1,g\rrbracket}$, and verify that, when Condition (1) or (2) is satisfied, it admits a non-zero term, hence is non-zero.
    
The assumption $s_1>s_0$ ensures that, for any $f\in\mathfrak S_J$ and $i\neq j$, the variable $X_{i,s_1-s_0}$ (resp. $X_{j,s_1-s_0}$) has degree exactly $1$ (resp. $0$) in $P_{i,s_1-s_0,f}$. It also ensures that $X_{i,s_1-s_0}$ has degree $p^{r(\alpha)}$ (resp. $0$) in $X_{f(0),s_1-s_0}^{p^{r(\alpha)}}\prod\limits_{k=1}^{g'}X^{p^{r_k}}_{f(k),0}=X_{f(0),s_1-s_0}^{p^{r(\alpha)}-1}P_{f(0),s_1-s_0,f}$ if $f(0)=i$ (resp. $f(0)\neq i$). 

Suppose first that Condition (1) is satisfied and let $r$ be as in Condition (1). Let $i\in J$, and let $P_{J,i,r}$ denote the component of $P_J$ made of the terms in which the variable $X_{i,s_1-s_0}$ has degree either $p^r$ or $1+p^r$. Since $r\neq 0$, we have $1\notin \{p^r,1+p^r\}$, and since the $P_\alpha$ are affine functions on the variables of level $s_1-s_0$, we deduce that
\begin{equation*}
    P_{J,i,r}=\bigg(X^{p^r-1}_{i,s_1-s_0}\sum_{\alpha\in \mathcal B_{s_0},r(\alpha)=r}\tilde{P}_\alpha\bigg)\sum_{f\in\mathfrak S_J, f(0)=i}P_{i,s_1-s_0,f}.
\end{equation*} This is non-zero by Condition (1) and Lemma~\ref{Lemma: equivalence of monomials}, hence $P_J$ is non-zero.

Suppose now that Condition (2) is satisfied. We consider $P_{J,i}$ the component of $P_J$ in which $X_{i,s_1-s_0}$ has degree $1$, and $X_{j,s_1-s_0}$ has degree $0$ for all $j\neq i$.

Let $P_\alpha$ denote the component of $\tilde{P}_\alpha$ which is constant with respect to every variable of level $s_1-s_0$.  Then decomposing $\tilde P_\alpha X_{f(0),s_1-s}^{p^{r(\alpha)}}\prod_{k=1}^{g'}X^{p^{r_k}}_{f(k),0}$ as a polynomial on these variables, its component occurring in $P_{J,i}$ is $P_\alpha P_{i,s_1-s_0,f}$ if $s=s_0$, $r(\alpha)=0$, $f(0)=i$, and it is $0$ otherwise.

 So the component $P_{J,i}$ of $P_J$ writes out: 
\begin{equation*}
    P_{J,i}=\sum\limits_{f\in \mathfrak S_J}\epsilon(f)\delta_{i,l}^{f(0)}P_{i,s_1-s_0,f}-\sum\limits_{\alpha\in \mathcal B_{s_0}, r(\alpha)=0}P_\alpha\sum\limits_{f\in \mathfrak S_J, f(0)=i} \epsilon(f)P_{i,s_1-s_0,f}. 
\end{equation*}

We deduce that
\begin{equation}\label{eq P alpha formula}
	P_{J,i}=0\iff \sum_{\alpha\in \mathcal B_{s_0}, r(\alpha)=0} P_\alpha = \frac{\sum\limits_{f\in \mathfrak S_J}\epsilon(f)\delta_{i,s_1-s_0}^{f(0)}P_{i,s_1-s_0,f}}{\sum\limits_{f\in \mathfrak S_J, f(0)=i} \epsilon(f)P_{i,s_1-s_0,f}}.
\end{equation}
For $i\neq k$ two elements of $J$, let $$Q_i^k:=\frac{\sum\limits_{f\in \mathfrak S_J,f(0)=k}\epsilon(f)P_{i,s_1-s_0,f}}{\sum\limits_{f\in \mathfrak S_J, f(0)=i} \epsilon(f)P_{i,s_1-s_0,f}}.$$

This is a function in the variables $(X_{j,0})_{j\in J}$.
The right-hand term in  Equation~\eqref{eq P alpha formula}  rewrites 
\begin{equation}\label{eq P alpha formula 2}\frac{\sum\limits_{f\in \mathfrak S_J}\epsilon(f)\delta_{i,s_1-s_0}^{f(0)}P_{i,s_1-s_0,f}}{\sum\limits_{f\in \mathfrak S_J, f(0)=i} \epsilon(f)P_{i,s_1-s_0,f}}=\delta_{i,s_1-s_0}^i+\sum\limits_{k\in J\setminus\{i\}}\delta_{i,s_1-s_0}^{k}Q_i^k.\end{equation}

Note that, if $f(0)=i\neq k$, then $\operatorname{deg}_{X_{k,0}}(P_{i,s_1-s_0,f})=p^{r_{f^{-1}(k)}}$, whereas, if $f(0)=k$, $\operatorname{deg}_{X_{k,0}}(P_{i,s_1-s_0,f})=0$. Therefore, with $r_{\min}=\operatorname{min}(r_1,\dots,r_{g'})$,  \begin{align*}
    \operatorname{val}_{X_{k,0}}(Q_i^k) &=-p^{r_{\min}}<0,\\ \operatorname{val}_{X_{i,0}}(Q_i^k) &=p^{r_{\min}}>0, \\
    \operatorname{val}_{X_{j,0}}(Q_i^k) &=0 \, \text{ for } j\in J\setminus \{i,k\}.
\end{align*}
In particular, from Equations~\eqref{eq P alpha formula} and~\eqref{eq P alpha formula 2} we deduce that, if $P_{J,i}=0$, the constant term of $\sum_{\alpha\in \mathcal B_{s_0}, r(\alpha)=0} P_\alpha$ is $\delta_i^i$. 

Now assume that $P_{J_,i}=0$ for all $i\in J$. Then all the $\delta_i^i$ are equal (and equal to $0$ if $0\notin r(\mathcal B_{s_0})$), and 
\begin{equation}\label{eq P alpha formula 3}
    \sum_{\alpha\in \mathcal B_{s_0}, r(\alpha)=0} P_\alpha-\delta_i^i=\sum\limits_{k\in J\setminus\{i\}}\delta^k_{i,s_1-s_0}Q^k_i.
\end{equation}

Assume by contradiction that $\delta_{i,s_1-s_0}^j\neq 0$ for some $j\neq i$ in $J$. Then by Equation~\eqref{eq P alpha formula 3}, $$\operatorname{val}_{X_{j,0}}(\sum_{\alpha\in \mathcal B_{s_0}, r(\alpha)=0} P_\alpha-\delta_i^i)\leq 0.$$ But, since we also suppose that $P_{J,j}=0$, Equation~\eqref{eq P alpha formula 3} also holds when replacing $i$ by $j$, and using $\delta_i^i=\delta_j^j$ we get $$\operatorname{val}_{X_{j,0}}(\sum_{\alpha\in \mathcal B_{s_0}, r(\alpha)=0} P_\alpha-\delta_i^i)= \operatorname{val}_{X_{j,0}}(\sum_{k\in J\setminus\{j\}}\delta^k_{j,s_1-s_0}Q^k_j)> 0,$$ a contradiction. 

We deduce that if $P_{J,i}=0$ for all $i\in J$ then $(\delta_{i,s_1-s_0}^j)_{(i,j)\in J^2}$ is diagonal. By contrapositive and by varying $J$ over all subsets of $\llbracket1,g\rrbracket$ of cardinal $g'+1$, this concludes the proof.
\end{proof}

We can now prove Proposition~\ref{proposition: U delta density max case}, and thus conclude the proof of Theorem~\ref{Theorem: generic automorphism group}.
\begin{proof}[Proof of Proposition~\ref{proposition: U delta density max case}]
If $|\mathcal{B}_{s_0}|<g$, these are particular cases of Proposition~\ref{proposition: U delta density}, so we suppose that $|\mathcal{B}_{s_0}|=g$.

Let $s_1>s_0$ be minimal such that $|\mathcal B_{s_1}|<g$.  Note that $s_1$ exists because $s_0<N$. Since for any $s$, $|\mathcal B_s|\leq |\mathcal B_{s+m}|$, we have $s_1-s_0 < m<n$. Therefore $\Lambda_{s_0}/\Lambda_{s_1+1}$ is a $p$-torsion $\breve{\Z}_p$-module, hence naturally an $\overline{\F}_p$-vector space with basis $([e_{i,s}]_{s_1+1})_{i\in \llbracket 1,g\rrbracket,s\in \llbracket s_0,N-1\rrbracket}$.

Let $(\delta_{i,l}^{j,s})$ denote the $\overline{\F}_p$-matrix coefficients of $[\delta]_{s_1+1}$, explicitly defined by $$[\delta(e_{i,l})]_{s_1+1}=\sum\limits_{j=1}^g\sum\limits_{s=s_0}^{s_1}\delta_{i,l}^{j,s}[e_{j,s}]_{s_1+1}.$$

	The projection $[\cdot]_{s_1}:\Lambda_{s_0}/\Lambda_{s_1+1}\rightarrow\Lambda_{s_0}/\Lambda_{s_1}$ together with the section mapping $[e_{j,s}]_{s_1}$ to $[e_{j,s}]_{s_1+1}$ induces a decomposition $$\Lambda_{s_0}/\Lambda_{s_1+1}=\Lambda_{s_0}/\Lambda_{s_1}\oplus \Lambda_{s_1}/\Lambda_{s_1+1}.$$

  On the one hand, by assumption on $s_1$ and by Lemma~\ref{Lemma: basis of level s}, for any $\underline{a}\in U(\overline{\F}_p)$, the vector space $\Lambda_{s_0}/\Lambda_{s_1}$ is freely generated by the vectors $(
  [F^aV^b(v(\underline{a}))]_{s_1}$ where $ s\in \llbracket s_0,s_1-1\rrbracket$ and $(a,b)\in \mathcal B_s$. 
  
  Since, for $(a,b)\in\mathcal B_s$, $F^aV^b(e_{i,l})=e_{i,l+s}$, the projected vectors $[F^aV^b(v(\underline{a}))]_{s_1}$ only depend on the variables $a_{i,l}$ of level $l<s_1-s\leq s_1-s_0$.  The truncation $[\delta(v(\underline{a}))]_{s_1}$ depends linearly on $\underline{a}$. So writing $$[\delta(v(\underline{a}))]_{s_1}=\sum\limits_{s=s_0}^{s_1-1}\sum\limits_{(a,b)\in \mathcal B_s} P_{(a,b)}(\underline{a})[F^a V^b(v(\underline{a}))]_{s_1},$$
    the functions $P_{(a,b)}$ are affine on the variables of level $l\geq s_1-s_0$.

For any $s\in \Z_{\geq 0}$ and $\alpha=(a,b)\in\mathcal B_s$, set $r(\alpha)=r(a,b)=a-b$. It is the integer such that $F^aV^b$ is $\sigma^{r(\alpha)}$-linear. Note that, restricted to a given $\mathcal B_s$, the map $r|_{\mathcal B_s}$ is injective, and that $0\in r(\mathcal B_s)$ if and only if $n\mid s$. We then have for $(a,b)\in \mathcal B_s$: 
$$[F^aV^b(v(\underline{a}))]_{s_1+1}=[F^aV^b(v(\underline{a}))]_{s_1} + \sum\limits_{j=1}^g a_{j,s_1-s}^{p^{r(a,b)}}[e_{i,s_1}]_{s_1+1}.$$ 
Therefore \begin{align*}&\bigg[\delta(v(\underline{a}))-\sum\limits_{s=s_0}^{s_1-1}\sum\limits_{(a,b)\in \mathcal B_s} P_{(a,b)}(\underline{a})F^a V^b(v(\underline{a}))\bigg]_{s_1+1}\\  = \quad & \sum\limits_{j=1}^g \bigg(\sum\limits_{l=0}^{N-1}\sum\limits_{i=1}^g\delta_{i,l}^{j,s_1} a_{i,l}-\sum\limits_{s=s_0}^{s_1-1}\sum\limits_{(a,b)\in \mathcal B_s} P_{(a,b)}(\underline{a}) a_{j,s_1-s}^{p^{r(a,b)}}\bigg)[e_{j,s_1}]_{s_1+1}.\end{align*}

    On the other hand, set $g'=|\mathcal B_{s_1}|$ and let $r_1,\dots,r_{g'}$ be a labeling of $\{r(a,b)\mid (a,b)\in \mathcal B_{s_1}\}$. By Lemma~\ref{Lemma: basis of level s} again, for $\underline{a}\in U(\overline{\F}_p)$ the vector subspace $[\D\cdot v(\underline{a})\cap \Lambda_{s_1}]_{s_1+1}\subset \Lambda_{s_1}/\Lambda_{s_1+1}$ is freely generated by $([F^aV^b(v(\underline{a}))]_{s_1+1})_{(a,b)\in \mathcal B_{s_1}}=(\sum\limits_{i=1}^g a_{i,0}^{p^{r_k}}[e_{i,s_1}]_{s_1+1})_{k\in \llbracket1,g'\rrbracket}$.

	We deduce
    \begin{equation*}
        [\delta(v(\underline{a}))]_{s_1+1}\notin [\D\cdot v(\underline{a})]_{s_1+1}\iff \operatorname{rk}(A_2(\underline{a}))= g'+1.
    \end{equation*} This is satisfied if and only if $P_J(\underline{a})\neq 0$ for some $(g'+1)$-minor $P_J$ of $A_2$. 
    \begin{enumerate}
        \item If $[\delta]_{s_0+1}\neq 0$, then there is $(a,b)\in \mathcal{B}_{s_0}$ such that $P_{(a,b)}\neq 0$, and if $n\nmid s_0$ then $0\notin r(\mathcal B_{s_0})$. Since the function $r$ is injective on $\mathcal B_{s_0}$ this suffices to check that Condition (1) of Lemma~\ref{Lemma: Minors of delta matrix bis} is satisfied in this case.
        \item If $\tau_1\delta=\delta\tau_1$ the matrix coefficients of $\delta$ satisfy $\delta_{i,l+1}^{j,s+1}=(\delta_{i,l}^{j,s})^{\sigma^b}$ for all $i,j,l,s$. In particular $\delta(\Lambda_0)\subset \Lambda_{s_0}$ implies that $\delta_{i,l}^{j,s_0}=0$ for all $l>0$. If moreover $n|s_0$ and $[\delta]_{s_0+1}\notin [p^\frac{s_0}{n}\breve{\Z}_p]_{s_0+1}$, the matrix $(\delta_{i,0}^{j,s_0})_{i,j}$ is not diagonal, and thus by $\tau_1$-commutativity neither is $(\delta_{i,s_1-s_0}^{j,s_1})_{i,j}$. Therefore Condition (2) of Lemma~\ref{Lemma: Minors of delta matrix bis} is satisfied.
    \end{enumerate} 
    Either way, by Lemma~\ref{Lemma: Minors of delta matrix bis}, one of the $(g'+1)$-minors of $A_2$ is non-zero (as an element of $\mathcal F$). Therefore $\operatorname{rk}(A_2(\underline{a}))= g'+1$ is satisfied for all $\overline{\F}_p$-points of a Zariski open and dense subset of $\A_{g,N}$, hence of $U$, which concludes the proof.
\end{proof}

\section{Torsion in the generic automorphism group}\label{Section: generic torsion}
In this section, we build on the results from Section \ref{Section: automorphisms of lattices} to understand how torsion in the automorphism group of the universal $p$-divisible group over $\M$ looks like.

We keep the notation as in the previous sections, in particular we have integers $n \geq 3$ and $m \geq 1$ such that $\gcd(m,n)=1$. Recall we have fixed the division algebra $D$ over $\Q_p$ whose invariant is $\frac{m}{n}$, with ring of integers $\O_D \subset D$ and uniformizer $\pi \in \O_D$.

For $g \geq 1$ and $s \geq 1$, we are going to study torsion in the congruence subgroups 
\[V_s = 1+ \pi^s\Mat_g(\O_D) \subset \GL_g(\O_D),\]
which necessarily needs to be of $p$-power order because $V_1$ is a pro-$p$-group. We treat the case $g=1$ separately because in this case we can be more precise.

\begin{lemma}\label{Lemma: torsion g=1}
    Assume that $g=1$, i.e. $V_s=1+\pi^s\O_D$. If $p-1 \nmid n$, then $V_1$ is torsion-free. If $p-1 \mid n$, then all $p$-torsion in $V_1$ occurs in $V_{\frac{n}{p-1}} \setminus V_{\frac{n}{p-1}+1}$ and thus $V_s$ is torsion-free for all $s \geq \frac{n}{p-1}+1$.
\end{lemma}
\begin{proof}
    By induction, it suffices to consider the case of $p$-torsion. Thus assume $1+\pi^s\xi \in V_s\setminus V_{s+1}$, i.e. $\xi \in \O_D^\times$, and consider the expression
    \begin{equation}\label{eq p-power expansion}
        (1+\pi^s\xi)^p - 1= p\pi^s\xi + \binom{p}{2}(\pi^s\xi)^2 + \cdots + (\pi^s\xi)^p.
    \end{equation}
Comparing the valuations of each summand, we see that the first valuation $n+s$ is strictly smaller than all others, except possibly the last, which is $ps$. Hence $n+s \neq ps$ implies $\operatorname{val}((1+\pi^s\xi)^p-1) = \frac{\min(n+s,ps)}{n} < \infty$, hence $(1+\pi^s\xi)^p \neq 1.$
If $p-1 \nmid n$, we thus obtain that $V_1$ is torsion-free. If $p-1 \mid n$, then all $p$-torsion in $V_1$ occurs in $V_{\frac{n}{p-1}}$.
\end{proof}

For later use, we make the following special case even more precise. 
\begin{lemma}\label{Lemma: special case p=2}
    Assume $g=1$ and $p=2$. Then the $2$-torsion in $V_1$ consists of $\{\pm 1\}$. Furthermore, any primitive $4$-torsion in $V_1$ occurs in $V_{\frac{n}{2}} \setminus V_{\frac{n}{2}+1}$. In particular, there is no primitive $4$-torsion in $V_1$ if $n$ is odd.
\end{lemma}
\begin{proof}
    From the previous lemma, we know that any $2$-torsion in $V_1$ occurs in $V_n$. So let $1 + \pi^n\xi=1+p\xi \in V_n \setminus V_{n+1}$ be $2$-torsion. Then
    \begin{equation*}
        0=(1+p\xi)^2 -1 = 2p\xi + (p\xi)^2 = p^2\xi+ p^2\xi^2
    \end{equation*}
    which implies that $\xi \in \{0,-1\}$, so the $2$-torsion in $V_1$ is as claimed.
    
    Now assume that $1+ \pi^s\xi \in V_s\setminus V_{s+1}$ is $4$-torsion but not $2$-torsion, i.e.~$(1+ \pi^s\xi)^2=-1\equiv 1+p\pmod{p\pi}$. On the other hand $(1+ \pi^s\xi)^2=1+2\pi^s \xi+\pi^{2s}\xi^2\equiv 1+\pi^{2s}\xi^2\pmod{p\pi}$. Thus $s = \frac{n}{2}$.
\end{proof}

\begin{lemma}\label{Lemma: torsion g geq 2}
    Assume $g \geq 2$, i.e. $V_s = 1+\pi^s\Mat_g(\O_D)$. Then $V_s$ is torsion free for all $s\geq \lfloor \frac{n}{p-1}\rfloor+1$.
\end{lemma}
\begin{proof}
  Since $\GL_g(\O_D)$ is a pro-$p$-group, it again suffices via induction to treat the case of $p$-torsion. Assume $1+\pi^s\xi \in V_{p,s}\setminus V_{p,s+1}$, i.e.~$\xi \in \Mat_g(\O_D)$ does not lie in the two-sided ideal $(\pi)$ of $\Mat_g(\O_D)$ generated by $\pi$. We again consider the expression \eqref{eq p-power expansion}, but this time we cannot directly argue with valuations. Instead, we want to find the maximal $e =e(p,s) \geq 0$ such that $\pi^e$ divides
    \[(1+\pi^s\xi)^p-1-p\pi^s\xi = \sum_{i=2}^p \binom{p}{i} (\pi^s\xi)^i\]
    for all $\xi \notin (\pi)$. If $p =2$, this is $e(2,s) = 2s$ and for $p \geq 3$, it is $e(p,s) = \min \{n+2s,ps\}$. We claim that, for all $s \geq \lfloor \frac{n}{p-1}\rfloor+1$, we have $e(p,s) > n+s$. Assuming this, then for all such $s$
    \[\xi \notin (\pi) \overset{e(p,s) > n+s}{\implies} p\pi^s\xi \not\equiv 0 \operatorname{mod} \pi^{e(p,s)} \implies (1+\pi^s\xi)^p \neq 1 \implies V_s \text{ is torsion-free.}\]

    We prove the claim by case distinction. If $p=2$, then $e(2,s) = 2s > n+s$ if and only if $s \geq n+1 = \lfloor\frac{n}{p-1}\rfloor + 1$. If $p \geq 3$ and $n \leq (p-2)s$, then $e(p,s) = n+2s$ and $e(p,s) > n+s$ holds for any $s \geq 1$. If $p \geq 3$ and $n > (p-2)s$, then $e(p,s) = ps$ and $e(p,s) > n+s$ holds if and only if $(p-2)s < n < (p-1)s$ if and only if $s \in (\frac{n}{p-1},\frac{n}{p-2})$. This proves the claim.
\end{proof}

The following proposition singles out those cases in which the congruence subgroup $V_N$ for $N= gm(n-m)-n+1$ might fail to be torsion-free.

\begin{proposition}\label{Proposition: V_N torsion-free}
    Assume that $n \geq 3$, then the congruence subgroup $V_N$ is torsion-free unless one of the following conditions is satisfied.
    \begin{enumerate}
        \item $(g,p)=(2,2)$ and $m\in \{1,n-1\}$.
        \item $g=1$ and $m\in \{1,n-1\}$
        \item $(g,p)=(1,2)$ and either $m \in \{2,n-2\}$ or $n=7$ and $m \in \{3,4\}$ or $n=8$ and $m \in \{3,5\}$.
    \end{enumerate}
    Furthermore, in Case (3) the torsion elements of $V_N$ are precisely $\{\pm 1\}$.
\end{proposition}
\begin{proof}
    By Lemmas \ref{Lemma: torsion g=1} and \ref{Lemma: torsion g geq 2}, $V_N$ is torsion-free once the inequality
    \begin{equation}\label{eq torsion inequality}
        gm(n-m) \geq n+\lfloor\frac{n}{p-1}\rfloor
    \end{equation}
    is satisfied. If $g=1$ and $p-1 \nmid n$, it suffices to consider the inequality $m(n-m) \geq n$ instead, which holds if and only if $m \notin \{1,n-1\}$.

    Note that \eqref{eq torsion inequality} holds regardless of $p$, $n$ and $m$ if $g \geq 3$. Indeed, in this case already $gm(n-m) \geq 2n$ because we assumed $n \geq 3$. Thus we assume that $g \leq 2$.

    First assume $g =2$ and $p \geq 3$. For $n \geq 4$,  \eqref{eq torsion inequality} holds because the stronger inequality
    \begin{equation*}
        2m(n-m) \geq \frac{3n}{2}
    \end{equation*}
    is satisfied. If $n=3$ and without loss of generality $m=1$, one checks by hand that \eqref{eq torsion inequality} is also satisfied. If on the other hand $p=2$, \eqref{eq torsion inequality} turns into $m(n-m) \geq n$, which holds whenever $m \notin \{1,n-1\}$.

    Now assume $g =1$. If $p-1 \nmid n$, then by Lemma \ref{Lemma: torsion g=1} we may instead of \eqref{eq torsion inequality} consider the inequality $m(n-m) -n \geq 0$, which holds whenever $m \notin \{1,n-1\}$. So assume that $p-1 \mid n$, in which case \eqref{eq torsion inequality} turns into
    \begin{equation}\label{eq torsion inequality g=1}
        m(n-m) \geq n+ \frac{n}{p-1}.
    \end{equation}
    We may without loss of generality assume $m \notin \{1,n-1\}$. The possibly stronger inequality $m(n-m) \geq 2n$ holds whenenever $n \geq 9$, given our assumptions on $m$ and $n$. This already treats all the cases listed above for $p=2$, so we are left to show that, for $n \leq 8$, $p \in \{3,5,7\}$ and all the other assumptions on $n$ and $m$ that we have already collected, inequality \eqref{eq torsion inequality g=1} holds.
    
    If $p=3$, the only possible combination left is $n=8$ and $m=3$, in which case the inequality $3\cdot 5 \geq 8+4$ holds. If $p=5$, the only possible combination left is also $n=8$ and $m=3$, in which case the inequality $3 \cdot 5 \geq 8+2$ also holds. If $p=7$, we would necessarily have $n = 6$, but then there is no $m \notin \{1,5\}$ that is coprime to $n$. This finishes the proof of the case distinction.

    Now assume we are in Case (3). By Lemma \ref{Lemma: special case p=2}, the $2$-torsion of $V_1$ lies in $V_n$ and consists of $\{\pm 1\}$. In Case (3), we have $N \leq n$ so that $V_n \subset V_N$, i.e. the $2$-torsion of $V_N$ is $\{\pm 1\}$. We claim that $V_N$ does not contain any primitive $4$-torsion. Indeed, by Lemma  \ref{Lemma: special case p=2} all primitive $4$-torsion lies in $V_{\frac{n}{2}}$ if $n$ is even, and otherwise there is none. Case (3) implies that $n \geq 5$, and for $n=5$ and $n=7$ there is no primitive $4$-torsion. If $n=6$, we are necessarily in Case (2). For $n=8$, the inequality $3\cdot5-8 \geq 4$ holds and by Lemma \ref{Lemma: special case p=2} there is no primitive $4$-torsion in $V_N$. Similarly for $n \geq 9$ and without loss of generality $m =2$ the inequality $2(n-2)-n = 2n-4 \geq \frac{n}{2}$ holds. This concludes the proof.
\end{proof}

\begin{lemma}\label{Lemma: torsion in 1+pi^N-1 O_D}
    Assume that $n \geq 3$ and $n \mid g$. Then the torsion of the congruence subgroup $1+\pi^{N-1}\O_D \subset \GL_g(\O_D)$ is contained in $\{\pm 1\}$, and is in fact trivial unless $(g,p)=(3,2)$ and $n=3$.
\end{lemma}
\begin{proof}
    By our assumptions $g \geq 3$. By Lemma \ref{Lemma: torsion g=1}, $1+\pi^{N-1}\O_D$ is always torsion-free if $p-1 \nmid n$, and if $p-1 \mid n$ it is torsion-free once $N-1 = gm(n-m)-n \geq \frac{n}{p-1}+1$. Because $g \geq 3$, this inequality is satisfied whenever $n \geq 4$, so let us assume that $n=g=3$, and hence without loss of generality $m=1$ and $p=2$. In this case, we may use Lemma \ref{Lemma: special case p=2} to see that $1+\pi^3\O_D$ contains all $2$-torsion, but no primitive $4$-torsion since all of the latter lies in $1+\pi^2\O_D$.
\end{proof}

\begin{proposition}\label{Proposition: generic torsion in p-div automorphism}
    Unless in Cases (1) or (2) of Proposition \ref{Proposition: V_N torsion-free}, the torsion of the automorphism group of the universal $p$-divisible group on the Rapoport-Zink space $\M$ is generically contained in $\Z_p^\times$.
\end{proposition}
\begin{proof}
    By all our reduction steps from Section \ref{Section: RZ space}, it suffices to show that generically on $\mathcal C_{\Lambda_0}$, the torsion is as claimed. By Theorem \ref{Theorem: generic automorphism group}, there is a dense open subscheme $Y_{\Lambda_0} \subset \mathcal C_{\Lambda_0}$ such that the automorphism group of the universal $p$-divisible group at each closed point is $\Gamma_{\gen} = \Z_p^\times + p^{\lceil\frac{N-1}{n}\rceil}\O_D + \pi^N\Mat_g(\O_D)$. Let $\gamma \in \Gamma_{\gen}$ be a torsion element. By multiplying with a suitable element of $\Z_p^\times$, we may assume $\gamma \in 1+p^{\lceil \frac{N-1}{n}\rceil}\O_D + \Mat_g(\O_D)$. Using the exact sequence
    \[0 \longrightarrow 1+\pi^N\Mat_g(\mathcal O_D) \longrightarrow \Gamma_{\gen} \longrightarrow [\Z_p^\times + p^{\lceil\frac{N-1}{n}\rceil}\O_D]_N \longrightarrow 0\]
    and Proposition \ref{Proposition: V_N torsion-free}, we find that the torsion subgroup of $\Gamma_{\gen}$ is contained in $\Z_p^\times$ if $n \nmid g$, and otherwise it is contained in $\Z_p^\times + p^{\lceil \frac{N-1}{n}\rceil}\O_D$. In the second case, we consider the exact sequence
    \[0 \longrightarrow 1+\pi^{N-1}\mathcal O_D \longrightarrow \Z_p^\times + \pi^{N-1}\O_D \longrightarrow \Z_p^\times \longrightarrow 0\]
    and use Lemma \ref{Lemma: torsion in 1+pi^N-1 O_D} to see that all torsion in $\Z_p^\times + \pi^{N-1}\O_D$ is contained in $\Z_p^\times$.
\end{proof}

By Lemma \ref{Lemma: comparison automorphism groups}, Proposition \ref{Proposition: generic torsion in p-div automorphism} also determines the generic torsion of the automorphism group of the universal $p$-divisible group over $\M_G$. This finally allows us to prove Theorem \ref{thmmain1}.

\begin{proof}[Proof of Theorem \ref{thmmain1}]
    Let $(A,i,\lambda)$ be a $k$-valued point of $\cS$, where $k$ is any field of characteristic $p$. Because $A$ and all of its endomorphisms are already defined over a subfield that is finitely generated over $\F_p$, by \cite[Theorem 2.6]{deJong1998Homomorphisms} there is an injection $\End_k(A) \otimes_\Z \Z_p \hookrightarrow \End_k(A[p^\infty])$. Due to the endomorphism structure on $A$, there is the following commutative diagram:
\[\begin{tikzcd}
	{\End_k(A)} & {\End_k(A)\otimes_\Z \Z_p} & {\End_k(A[p^\infty])} \\
	{\O_L} & {\O_L\otimes_\Z \Z_p} & {\O_L \otimes_\Z \Z_p}
	\arrow[hook, from=1-1, to=1-2]
	\arrow[hook, from=1-2, to=1-3]
	\arrow[hook, from=2-1, to=1-1]
	\arrow[hook, from=2-1, to=2-2]
	\arrow[hook, from=2-2, to=1-2]
	\arrow["{=}", from=2-2, to=2-3]
	\arrow[hook, from=2-3, to=1-3]
\end{tikzcd}\]
    Let $f \in \End_k(A)$ such that $f[p^\infty] \in \O_L \otimes_\Z \Z_p$. We see that first of all $f \otimes 1 \in \O_L \otimes_\Z \Z_p$ as an element of $\End_k(A) \otimes_\Z \Z_p$, by injectivity of the upper right map. Furthermore both $\End_k(A)$ and $\O_L$ are finite free $\Z$-modules of rank $2\dim(A)$ resp. $2$, showing that necessarily $f \in \O_L$.

    Note that our assumptions in Theorem \ref{thmmain1} are precisely the assumptions in Proposition \ref{Proposition: generic torsion in p-div automorphism}. Thus we can use Proposition \ref{Proposition: generic torsion in p-div automorphism} and Lemma \ref{Lemma: comparison automorphism groups} to conclude that there is an open dense subspace of $\M_G$ where the torsion subgroup of the automorphism group of the universal $p$-divisible group is contained in $(\Z_p \times \Z_p)^\times = (\O_L \otimes_\Z \Z_p)^\times$. Via $p$-adic uniformization (see Section \ref{subsection: p-adic uniformization}) we find an open dense subscheme $Y_{b_0}$ of $\cS_{b_0}$ such that for every $x \in Y_{b_0}(\overline{\F}_p)$
    \[\Aut(A_x,i_x,\lambda_x) \hookrightarrow \Aut(A_x[p^\infty],i_x[p^\infty],\lambda_x[p^\infty])_{\tors} \subseteq (\Z_p \times \Z_p)^\times = (\O_L \otimes_\Z \Z_p)^\times,\]
    By the above discussion, this implies that $\Aut(A_x,i_x,\lambda_x) \subseteq \O_L^\times$ for all $x \in Y_{b_0}(\overline{\F}_p)$.

    Recall that $\O_L^1$ denotes the subgroup of elements $\alpha \in \O_L^\times$ such that $\Norm(\alpha) = \alpha^*\alpha = 1$. Any $\alpha \in \O_L^\times$ such that $i_x(\alpha)$ commutes with $\lambda_x$ necessarily satisfies $\Norm(\alpha)=1$. On the other hand, every $\alpha \in \O_L^1$ defines an automorphism $i(\alpha) \in \Aut(A_x,i_x,\lambda_x)$. Hence for every $x \in Y_{b_0}(\overline{\F}_p)$
    \[\Aut(A_x,i_x,\lambda_x) = \O_L^1,\]
    which finishes the proof.
\end{proof}

\begin{example}\label{example: exceptional cases}
    Since $L$ is imaginary quadratic, by Dirichlet's unit theorem $\O_L^\times = \mu_L$, where $\mu_L$ denotes the roots of unity in $L$. The only two exceptional cases where there are more roots of unity than just $\{\pm 1\}$ are precisely the ones where $\O_L^1 \neq \{\pm 1\}$, namely $L= \Q(\zeta_4)$ and $L=\Q(\zeta_6)$.
    
    Note that in any cyclotomic extension $\Q(\zeta_n)$, a prime $p$ splits completely if and only if $p \equiv 1 \mod n$. Thus in our two exceptional cases, the prime $p$ being split implies that $\Z_p^\times$ contains all $4$-th resp. $6$-th roots of unity. By Lemma \ref{Lemma: comparison automorphism groups}, the same is true for the corresponding subgroup of $\Aut_G(X)$ for any $p$-divisible group $X$ with $G$-structure. In all other cases, $\Z_p^\times$ of course still contains all $(p-1)$-th roots of unity, but in $(\O_L \otimes \Z_p)^\times$ they are purely of $p$-adic nature and thus don't contribute to automorphisms of abelian varieties. 
    
    Also note that under the identification $(\O_L \otimes \Z_p)^\times = \Z_p^\times \times \Z_p^\times$, $\zeta_4$ resp. $\zeta_6$ indeed gets sent into the subgroup
    \[\{(a,a^{-1})\mid a \in \Z_p^\times\}.\]
\end{example}

\begin{remark}\label{remark: remaining cases}
    If the signature $(m_0,n_0-m_0)$ is not as in the statement of Theorem \ref{thmmain1}, we expect the generic automorphism group of the universal abelian variety over the basic locus to be larger than $\O_L^1$. On the other hand, in Case (1) of Proposition \ref{Proposition: V_N torsion-free}, we have $N= n-1$, thus the only torsion appearing in $V_N$ lies in $V_n$ and needs to be $2$-torsion. Thus any of its finite subgroups is isomorphic to $(\Z/2\Z)^k$ for some $k \geq 1$.
\end{remark}

\bibliographystyle{alpha}
\bibliography{bibliography.bib}

\end{document}

%% file: preamble.tex
\usepackage{amsmath, amssymb, amsthm, mathrsfs} 
\usepackage{tikz-cd}
\usepackage[a4paper, total={6in, 8in}]{geometry}
\usepackage{stmaryrd}
\usepackage{xcolor}
\usepackage{dsfont}
\usepackage{hyperref}

\newtheorem{theorem}{Theorem}[section]
\newtheorem{lemma}[theorem]{Lemma}
\newtheorem{proposition}[theorem]{Proposition}

\newtheorem{notation}[theorem]{Notation}
\newtheorem*{theorem*}{Theorem}

\theoremstyle{definition}
\newtheorem{definition}[theorem]{Definition}
\newtheorem{example}[theorem]{Example}

\theoremstyle{remark}
\newtheorem{remark}[theorem]{Remark}

\def\A              {\mathbb{A}}

\def\cA             {\mathcal{A}}
\def\ccA            {\mathscr{A}}

\def\C              {\mathbb{C}}
\def\D              {\mathcal{D}}

\def\F              {\mathbb{F}}

\def\G              {\mathbb{G}}

\def\cH             {\mathcal{H}}

\def\Q              {\mathbb{Q}}

\def\R              {\mathbb{R}}

\def\cS             {\mathcal{S}}

\def\Z              {\mathbb{Z}}
\def\O              {\mathcal{O}}
\def\M              {\mathcal{M}}

\def\X              {\mathbb{X}}

\def\iso            {\cong}

\def\Aut            {\operatorname{Aut}}

\def\Hom            {\operatorname{Hom}}

\def\diag           {\operatorname{diag}}

\def\End            {\operatorname{End}}

\def\Gal            {\operatorname{Gal}}
\def\gen            {\operatorname{gen}}
\def\GL             {\operatorname{GL}}

\def\GSp            {\operatorname{GSp}}
\def\GU             {\operatorname{GU}}
\def\id             {\operatorname{id}}

\def\ker            {\operatorname{ker}}

\def\Lie            {\operatorname{Lie}}

\def\Mat            {\operatorname{Mat}}

\def\Norm           {\operatorname{Nm}}

\def\opp            {\operatorname{op}}

\def\Sh             {\operatorname{Sh}}

\def\Spf            {\operatorname{Spf}}
\def\Spec           {\operatorname{Spec}}
\def\swap           {\operatorname{swap}}

\def\Tr             {\operatorname{Tr}}
\def\tors           {\operatorname{tors}}